\documentclass[11pt,reqno]{amsart}

\usepackage{geometry}
\usepackage[utf8]{inputenc}
\usepackage{amstext,latexsym,amsbsy,amsmath,amssymb,amsthm,mathtools,relsize,geometry,enumerate}
\usepackage{hyperref}
\hypersetup{
  colorlinks   = true, %Colours links instead of ugly boxes
  urlcolor     = blue, %Colour for external hyperlinks
  linkcolor    = red, %Colour of internal links
  citecolor   = red %Colour of citations
}
\theoremstyle{definition}
\usepackage{graphicx}
\setkeys{Gin}{width=\linewidth,totalheight=\textheight,keepaspectratio}
\graphicspath{{./images/}}

\usepackage{multicol}
\usepackage{booktabs}
\usepackage{mathrsfs}\usepackage{xcolor}
\usepackage[normalem]{ulem}
\usepackage{cancel}

\newcommand{\rbb}{\mathbb{R}}

\newcommand{\Fcal}{\mathcal{F}}

\newcommand{\Sc}{\mathcal{S}}

\newcommand{\Kcos}{\mathcal{K}_{\cos}}
\newcommand{\Ksin}{\mathcal{K}_{\sin}}
\newcommand{\rhat}{\widehat{r}}

\newcommand{\bbR}{\mathbb{R}}

\newcommand{\bbC}{\mathbb{C}}
\newcommand{\bbE}{\mathbb{E}}

\renewcommand{\d}{\mathrm{d}}

\newcommand{\f}{\varphi}

\newcommand{\E}[1]{\mathbb{E}\left[#1\right]}
\newcommand{\Enone}[1]{\mathbb{E}#1}

\newcommand{\close}{\!\!\!}

\numberwithin{equation}{section}
\theoremstyle{definition}
\theoremstyle{plain}
\newtheorem{theorem}{Theorem}[section]

\newtheorem{lemma}[theorem]{Lemma}
\newtheorem{assumption}[theorem]{Assumption}
\newtheorem{proposition}[theorem]{Proposition}
\newtheorem{definition}[theorem]{Definition}

\newtheorem{remark}[theorem]{Remark}

\numberwithin{equation}{section}

\title[Asymptotic analysis of MSD]{Asymptotic analysis of the mean squared displacement under fractional memory kernels}

\author[G.~Didier]{Gustavo~Didier$^1$}
\author[H.~Nguyen]{Hung~D.~Nguyen$^2$}

\thanks{\noindent \hspace{-0.52cm} $^1$ Department of Mathematics, Tulane University, New Orleans, LA 70118, USA}

\thanks{\noindent $^2$ Department of Mathematics, University of California, Los Angeles, CA, 90095, USA}

\calclayout

\begin{document}
\begin{abstract}
The generalized Langevin equation (GLE) is a universal model for particle velocity in a viscoelastic medium. In this paper, we consider the GLE family with fractional memory kernels. We show that, in the critical regime where the memory kernel decays like $1/t$ for large $t$, the mean squared displacement (MSD) of particle motion grows linearly in time up to a slowly varying (logarithm) term. Moreover, we establish the well-posedness of the GLE in this regime. This solves an open question from~\cite{mckinley2018anomalous} and completes the answer to the conjecture put forward in~\cite{morgado2002relation} on the relationship between memory kernel decay and anomalously diffusive behavior. Under slightly stronger assumptions on the memory kernel, we construct an Abelian-Tauberian framework that leads to robust bounds on the deviation of the MSD around its asymptotic trend. This bridges the gap between the GLE memory kernel and the spectral density of anomalously diffusive particle motion characterized in \cite{didier2017asymptotic}.

%We consider the generalized Langevin equation, which is commonly used to model particle movements in viscoelastic fluids, with power-law decay memories. We show that when the memory tail decays similar to $1/t$ as $t$ tends to infinity, the mean-squared displacement (MSD) of the particle grows nonlinearly by a $\log$ scaling in time. The result answers an open question in~\cite{mckinley2018anomalous} and a conjecture proposed previously in~\cite{morgado2002relation}. With further assumptions on the decaying rate of the memory, we obtain the growing rate of the MSD.
\end{abstract}
\maketitle

%%\pacs{1315, 9440T}

\section{Introduction}

The velocity of freely-moving microparticles embedded in viscous, Newtonian fluids is classically modeled by means of a Langevin equation. %This involves the implicit assumption that there is no time correlation between the foreign microparticles and the thermally fluctuating fluid molecules .
However, unlike in a Langevin framework, the presence of elasticity in a non-Newtonian fluid induces time correlation between the foreign microparticle movement and molecular bombardment~\cite{didier2012statistical,didier2017asymptotic,kneller2011generalized,kubo1966fluctuation,mason1995optical,mori1965transport}. The generalized Langevin equation (GLE) was introduced in~\cite{kubo1966fluctuation,mori1965transport} and later popularized in~\cite{mason1995optical} as a universal model for particle velocity in a viscoelastic medium. It is given by the one-dimensional stochastic-integro-differential equation \cite{chandler1987introduction,goychuk2012viscoelastic,hohenegger2017fluid,
indei2012treating,mckinley2018anomalous,zwanzig2001nonequilibrium}
\begin{equation}\label{eq:gle:full}
m\, \dot{V}(t)	=  -\gamma V(t)- \beta \int_{-\infty}^t \!\!\!\! K(t-s) V(s) \d s + \sqrt{\beta} F(t) \d t +\sqrt{2\gamma}\dot{W}(t).
\end{equation}
In \eqref{eq:gle:full}, $m$ is the particle's mass, $\gamma$ and $\beta$ are, respectively, the viscous and elastic drag coefficients, $K(t)$ is the memory kernel that reflects the drag impact of the surrounding media on the particle over time, and $W(t)$ is the standard Brownian motion. The term $F(t)$ is a stationary, Gaussian process satisfying the so-named fluctuation-dissipation relationship
\begin{equation} \label{eqn:fluctuation-dissipation}
\E{F(t)F(s)}=K|t-s|,
\end{equation}
a balance-of-force condition originally formulated in~\cite{kubo1966fluctuation,onsager1931reciprocal}. %For a more detailed derivation of the GLE in literature, we refer to for example~\cite{chandler1987introduction,goychuk2012viscoelastic,indei2012treating,zwanzig2001nonequilibrium}.

%One of the reasons that the GLE has been widely used is because of a distinct behavior that is called \emph{anomalous diffusion}.

The GLE is a model of \textit{anomalous diffusion}, a topic that has been the focus of intensive research efforts in the modern biophysical literature (e.g., \cite{saxton1994anomalous,saxton1996anomalous,levine2000one,metzler:klafter:2000,sokolov2008statistics,meroz:sokolov:2015,ghosh:cherstvy:grebenkov:metzler:2016}). The physical definition of anomalous diffusion is based on the behavior over time of the (ensemble) \textit{mean squared displacement} (MSD) $\E{X(t)^2}$ of the observed particle. More precisely, let $X(t)=\int_0^t V(s)\d s$ be the particle position process, where $V(t)$ is the particle velocity process in~\eqref{eq:gle:full}. Then, the particle is said to be asymptotically
\begin{align*}
\begin{cases}
\text{subdiffusive}, \\ \text{diffusive}, \\ \text{superdiffusive},
\end{cases}
\text{if}\,\,\E{X(t)^2}\sim t^\alpha\,\, \text{as}\,\, t\to\infty\,\,\text{for}\,\,
\alpha\begin{cases}\in(0,1),\\ =1,\\ \in(1,\infty),
\end{cases}
\end{align*}
where we write $f(t)\sim g(t)$ as $t\to\infty$ whenever $f(t)/g(t)\to c\in(0,\infty)$. While diffusion ($\alpha=1$) is usually observed in single particle tracking experiments in viscous fluids \cite{hill2014biophysical}, subdiffusion ($0<\alpha<1$) is often detected in viscoelastic fluids \cite{ernst2012fractional,golding2006physical,hill2014biophysical,zhang:crizer:schoenfisch:hill:didier:2018}.

Since the earliest formulations of the GLE, it was believed that the asymptotic behavior of the microparticle modeled by \eqref{eq:gle:full} is entirely determined by the tail decay of the memory kernel $K$, and that the GLE has subdiffusive solutions. This conjecture was formally proposed in~\cite{morgado2002relation} as
\begin{align}\label{conjecture:morgado}
\text{If there exists}\,\alpha>0\,\text{such that}\, K(t) \sim t^{-\alpha},\,\,
\text{then}\, \E{X(t)^2}\sim t^{\alpha}\,\text{ as}\,\, t\to\infty.
\end{align}
Several authors have tackled the issue of the connection between memory in particle behavior and the asymptotics of the MSD (e.g., \cite{didier2012statistical,kneller2011generalized,kupferman2004fractional}). To the best of our knowledge, the first rigorous results on \eqref{conjecture:morgado} were obtained in~\cite{kou2008stochastic} for the memory kernel instance $K(t)=t^{-\alpha}$, $\alpha\in(0,1)$. Using the explicit form of the associated Fourier transforms, the results confirm that the GLE solution exhibits subdiffusive behavior. More recently, it was shown under mild assumptions that, when $K$ is integrable, the solution of the GLE~\eqref{eq:gle:full} is diffusive; otherwise, if $K(t)\sim t^{-\alpha}$, $\alpha\in(0,1)$, the solution is subdiffusive~\cite{mckinley2018anomalous}. This corroborates the conjecture~\eqref{conjecture:morgado} for the parameter range $0<\alpha<1$, but disproves it for $\alpha>1$ since superdiffusion is unattainable.  %In other words, when $\alpha>1$ and $K(t)\sim t^{-\alpha}$ as $t\to\infty$, the memory kernel $K$ is integrable, which implies diffusion, not superdiffusion according to conjecture~\eqref{conjecture:morgado}.

More generally, for a wide range of physically inspired stochastic differential equations, much of the observed dynamics is based on the relationship between memory kernels and the asymptotic behavior of autocorrelation functions, as well as that of the MSD. Over the last few decades, several authors have established key results on this topic for a number of models such as the Kubo-Mori-Okabe-Langevin equation \cite{inoue1993equations,okabe1979stationary} and the Stokes-Boussinesq-Langevin equation \cite{okabe1986theory}. In \cite{anh2019fractional}, the autocorrelation function for the fractional Stokes-Boussinesq-Klein equation is shown to exhibit power law decay (see also \cite{metzler1999anomalous} and \cite{barkai2000fractional,metzler2000generalized} on the fractional Fokker-Planck and the fractional Klein-Kramers equations, respectively).

%\GDcomment{I just simplified the paragraph you wrote - note that, in an introduction, we don't need to provide details of SDEs that don't appear in our paper. I also wanted to avoid saying explicitly that those papers are ``just formal".} \GD{For a wide range of physically inspired equations related to the GLE, the relationship between memory kernels and the asymptotic behavior of autocorrelation functions, as well as that of the MSD, has been established by a number of authors over the last few decades. Examples include the KMO-Langevin equation \cite{inoue1993equations,okabe1979stationary} \GDcomment{What is KMO?} and the Stokes-Boussinesq-Langevin equation \cite{okabe1986theory}. Also, in \cite{anh2019fractional}, the power law decay of the autocorrelation function is shown for the case of fractional Stokes-Boussinesq-Klein equation (see also \cite{metzler1999anomalous} and \cite{barkai2000fractional,metzler2000generalized} on the fractional Fokker-Planck and the fractional Klein-Kramers equations, respectively).}

%Throughout this note, we will consider the following equation
In this paper, we focus on the distinctively viscoelastic features of \eqref{eq:gle:full} and consider the GLE family given by
\begin{equation} \label{eq:gle}
m\, \dot{V}(t)	=  - \beta \int_{-\infty}^t \!\!\!\! K(t-s) V(s) \d s + \sqrt{\beta} F(t) \d t ,
\end{equation}
%which is a reduced version of~\eqref{eq:gle:full} where $\gamma=0$. The first goal of this note is to fill the gap of the MSD behavior in the remaining
corresponding to $\gamma = 0$ in~\eqref{eq:gle:full} (see also Remark \ref{r:extension_full_GLE}). In the first set of main results, we tackle and solve the problem left open in~\cite{mckinley2018anomalous} by establishing the asymptotic growth rate of the MSD for the case where the memory kernel satisfies $K(t)\sim t^{-1}$ as $t\to\infty$. Because of its unique character, we call this regime \emph{critical}, in contrast with \emph{diffusive} and \emph{subdiffusive} regimes. Conjecture~\eqref{conjecture:morgado} suggests that, in this situation, the MSD grows linearly in time, i.e., $\E{X(t)^2}\sim t $ as $t\to\infty$. However, we show that the MSD is asymptotically linear only up to a slowly varying (logarithm) factor (Theorem~\ref{thm:asymptotic:critical}). Moreover, the peculiar tail behavior of the memory kernel in the critical regime requires Fourier analysis techniques that are different from those in~\cite{mckinley2018anomalous}. In particular, we draw upon an Abelian-type characterization of the memory kernel in the Fourier domain~\cite{inoue1995abel,pitman1968behaviour}. We further extend the broad framework developed in~\cite{mckinley2018anomalous} to establish the well-posedness of~\eqref{eq:gle} (Theorem \ref{thm:well-posed:critical}; see also Remark \ref{r:harmonizable}). The weak solutions are constructed based on the celebrated theory of stationary random distributions~\cite{ito1954stationary}, which is rather flexible and naturally well suited for the GLE framework. %To avoid tedious computation, we will only present the results for~\eqref{eq:gle}. Using a similar arguments, the same results also holds for the full equation~\eqref{eq:gle:full} in this \emph{critical} regime ($K(t)\sim t^{-1}$ as $t\to\infty$)

In the second set of main results, under slightly stronger assumptions we establish the relationship between the memory kernel decay rates and robust bounds on the deviation of the MSD around its asymptotic trend. The problem of characterizing the convergence rate of the MSD or its statistical counterpart, the time-averaged MSD (TAMSD), in different settings has been studied in many works (e.g., \cite{sokolov2008statistics,deng2009ergodic,jeon2010analysis,burnecki:muszkieta:sikora:weron:2012,vestergaard:blainey:flyvbjerg:2014,kepten:weron:sikora:burnecki:garini:2015,sikora:teuerle:wylomanska:grebenkov:2017}). For a fractional Brownian motion $\{B_H(t)\}_{t \in \bbR}$ (fBm), self-similarity leads to the MSD exhibiting \textit{exact} power law scaling $\bbE B_H(t)^2 = \sigma^2 |t|^{\alpha}$, where $\alpha = 2H$ and $H \in (0,1)$ is the so-called Hurst parameter \cite{embrechts:maejima:2002,pipiras:taqqu:2017}. In~\cite{didier2017asymptotic}, for a broad class of Gaussian, stationary increment processes, it is shown that the MSD scales like a power law asymptotically, and that its finite-time deviation from the fBm MSD is generally controlled by the relation
\begin{equation}\label{e:bound_ensemble_MSD}
\Big| \frac{ \E{X(t)^2} }{2D t^{\alpha}} - 1\Big| \leq \frac{C}{t^{\delta}}, \quad \textnormal{large $t$},
\end{equation}
for some diffusivity constant $D > 0$. In \eqref{e:bound_ensemble_MSD}, the deviation parameter $\delta > 0$ is mostly determined by the high frequency components of the particle's motion. Not only does the bound \eqref{e:bound_ensemble_MSD} provide a robust characterization of the MSD and its relation to self-similarity, but also it plays a key role in establishing the weak convergence of TAMSD-based statistics frequently used in biophysical data analysis (cf.\ \cite[Proposition 1 and Corollary 1]{didier2017asymptotic}). %As a consequence, the bound \eqref{e:bound_ensemble_MSD} is determined semiparametrically by the behavior of the spectral density around the origin. %introduced conditions on the spectral density of the concerned processes to provide the rate that their MSD converge to equilibrium.
However, it is not straightforward to translate the required conditions on the spectral density into conditions on the memory kernel of the GLE. In this paper, we tackle this problem and construct a comprehensive Abelian-Tauberian framework that bridges the gap between GLE memory kernel decay and relations of the type \eqref{e:bound_ensemble_MSD} (see \cite{leonenko2013tauberian,leonenko1999limit,yakimiv2005probabilistic,
bingham2008tauberian,leonenko2013tauberian} on the theory and applications of Abelian-Tauberian schemes in one or multiple dimensions). The results require mild conditions and cover all regimes, i.e., critical, diffusive and subdiffusive (Theorems \ref{thm:convergent-rate:diffusion+subdiffusion} and \ref{thm:convergent-rate:critical}).
 %In section~\ref{sec:results} below, we will establish sufficient conditions on $K(t)$ that can predict the \GD{convergence rate} of the MSD. %As mentioned in previous paragraph, one can also establish \GD{convergence rate} for the full equation~\eqref{eq:gle:full} using the same technique as that we are going to present for~\eqref{eq:gle}.

The rest of the paper is organized as follows. In Section~\ref{sec:results}, we state the assumptions and main results of the paper. In Section~\ref{sec:Fourier}, we lay out the Fourier analysis framework. In Section~\ref{sec:wellposed}, we address the well-posedness of~\eqref{eq:gle} in the critical regime. In Section~\ref{sec:MSD:critical}, we establish the asymptotic growth rate of the MSD in the critical regime under minimal assumptions on the memory kernel. In Section~\ref{sec:MSD:convergent-rate}, we construct the robust bounds for the deviation of the MSD around its asymptotic trend. Section~\ref{sec:discuss} contains conclusions and a discussion of open problems.

\section{Assumptions and main results} \label{sec:results}
For a given function $K:\rbb\to\rbb$, let $\Kcos$ and $\Ksin$ be the Fourier-type transforms of $K$ defined by
\begin{align}\label{e:Kcos_Ksin}
\Kcos(\omega) = \int_0^\infty\close K(t)\cos(t\omega)\d t,\qquad \Ksin(\omega)=\int_0^\infty\close K(t)\sin(t\omega)\d t,
\end{align}
where the integrals above are understood in the sense of improper integrals whenever they converge.

We assume the following conditions on the memory kernel.
\begin{assumption} \label{cond:K:general}
Let $K:\rbb\to \rbb \cup \{\infty\}$ be a memory kernel obtained from a solution to \eqref{eq:gle}, where $K$ may only be infinite at $t = 0$.
\begin{enumerate}[(I)]
\item \label{cond1} \begin{enumerate}[(a)]
\item \label{cond1a} $K \in L^1_{loc}(\rbb)$ is symmetric around zero and positive for all  non-zero $t$;
\item \label{cond1b} $K(t) \to 0$ as $t \to \infty$ and is eventually decreasing;
\item \label{cond1c}The improper integral $\Kcos(\omega)=\int_0^\infty \! K(t)\cos(\omega t) \, \d t$ is positive for all non-zero $\omega$.
\end{enumerate}

\item \label{cond:powerlaw:t}$K(t)\sim t^{-1}$ as $t\to\infty$.

%\item \label{cond3} there exists $c,\,\beta>0$ such that $|tK(t)-C|\leq c t^{-\beta}$, where $C:=\lim_{t\to\infty}t K(t)$.
\end{enumerate}

\end{assumption}
Conditions~\eqref{cond1a} and ~\eqref{cond1b} are quite standard when studying the asymptotic behavior of Fourier transforms. Also, they guarantee that $\Kcos(\omega)$ and $\Ksin(\omega)$ are well-defined for every non-zero $\omega$ as in Lemma~\ref{lem:fcos-fsin:well-defined}. Condition~\eqref{cond1c} may seem unusual, but we will see later in the proof of Theorem~\ref{thm:well-posed:critical} in Section~\ref{sec:wellposed} that it is required to guarantee the existence of stationary solutions for~\eqref{eq:gle}. Note that a sufficient condition for~\eqref{cond1c} to hold is that $K(t)$ be convex \cite{tuck2006positivity}.

We have not yet defined the notion of a solution of~\eqref{eq:gle}. As explained in the Introduction, in Section~\ref{sec:wellposed} we recap the well-posedness of the framework of~\cite{mckinley2018anomalous} and use it to formulate the concept of a weak solution of the GLE (Theorem \ref{thm:well-posed:critical}). Hence, for expositional purposes, we can simply assume a weak solution exists and state the first of the main results of the paper, which describes the asymptotic growth rate of the MSD in the critical regime. The proof of Theorem~\ref{thm:asymptotic:critical} is carried out in Section~\ref{sec:MSD:critical}.
\begin{theorem} \label{thm:asymptotic:critical} Suppose that $K(t)$ satisfies~\eqref{cond1} and~\eqref{cond:powerlaw:t}. Let $V$ be the weak solution of~\eqref{eq:gle} as in Theorem~\ref{thm:well-posed:critical} and let $X(t)$ be the position process associated with $V$ as in~\eqref{def:v(t)-process}. Then,
\begin{equation*}
\E{X(t)^2}\sim \frac{t}{\log (t)},\quad\text{as}\quad t\to\infty.
\end{equation*}
\end{theorem}

We now turn to the topic of bounds for the growth rate of the MSD. To establish these, we need stronger conditions, namely, we assume the memory kernel in each regime converges polynomially fast.
\begin{assumption} \label{cond:K:convergent-rate} Let $K$ be a memory kernel obtained from a solution to \eqref{eq:gle} and taking values in $[0,\infty)$ for $t > 0$. %Given $K:(0,\infty)\to[0,\infty)$, we assume that
\begin{enumerate}[(I)]
\addtocounter{enumi}{2}
\item \label{cond:K:convergent-rate:diffusion} Diffusive regime: $K\in L^1(0,\infty)$ and that there exists a positive $\beta_0>0$ such that
\begin{align}
t^{\beta_0} K(t)\in L^1(0,\infty);
\end{align}
\end{enumerate}
%(a) $K\in L^1(0,\infty)$ and that there exists a positive $\beta_0>0$ such that
%\begin{align} \label{cond:K:convergent-rate:diffusion}
%t^{\beta_0} K(t)\in L^1(0,\infty),
%\end{align}

\begin{enumerate}[(I)]
\addtocounter{enumi}{3}
\item \label{cond:K:convergent-rate:subdiffusion} Subdiffusive regime: there exist $\alpha\in (0,1)$, $C_\alpha >0$ and $\beta_{\alpha}>0$ such that $K(t)\sim t^{-\alpha}$ as $t\to\infty$ and that
\begin{align}\label{e:K:convergent-rate:subdiffusion}
|t^\alpha K(t)-C_\alpha|= O( t^{-\beta_{\alpha}}),\quad t\to\infty;
\end{align}
\end{enumerate}

\begin{enumerate}[(I)]
\addtocounter{enumi}{4}
\item \label{cond:K:convergent-rate:critical} Critical regime: $K(t)\sim t^{-1}$ as $t\to\infty$ and there exist $C_1 >0$ and $\beta_1>0$ such that
\begin{align}\label{e:K:convergent-rate:critical}
|tK(t)-C_1|= O( t^{-\beta_1}),\quad t\to\infty.
\end{align}
\end{enumerate}
\end{assumption}

\begin{remark}
Note that, under conditions~\eqref{cond:K:convergent-rate:diffusion} and ~\eqref{cond:K:convergent-rate:subdiffusion}, the well-posedness of~\eqref{eq:gle} is shown in~\cite{mckinley2018anomalous} under the same notion of weak solution put forth in Definition~\ref{def:weak-solution}.
\end{remark}

In the following theorem, we provide bounds for the MSD growth rate in the first two regimes described in Assumption \ref{cond:K:convergent-rate}, i.e., diffusive and subdiffusive. The proofs for these regimes are similar and make use of a careful analysis of the convergence rate of $\Kcos(\omega)$ and $\Ksin(\omega)$ as $\omega\to 0$ (see Section Section~\ref{sec:MSD:convergent-rate}).
\begin{theorem} \label{thm:convergent-rate:diffusion+subdiffusion}
Suppose that $K(t)$ satisfies~\eqref{cond1}. Let $V$ be the weak solution of~\eqref{eq:gle} as in Definition~\ref{def:weak-solution} and let $X(t)$ be the position process associated with $V$ as in~\eqref{def:v(t)-process}.

\noindent (a) If $K(t)$ satisfies condition~\eqref{cond:K:convergent-rate:diffusion}, then
\begin{align}\label{e:MSD_III}
\Big|\frac{\E{X(t)^2}}{t}- \frac{2}{\beta\Kcos(0)}\Big| = O(t^{-\gamma_0/2}), \quad t \rightarrow \infty,
\end{align}
where $\gamma_0=\min\{\beta_0,2\}$ and $\beta_0$ is the constant from~\eqref{cond:K:convergent-rate:diffusion}.

\noindent (b) If $K(t)$ satisfies condition~\eqref{cond:K:convergent-rate:subdiffusion}, then
\begin{align}\label{e:case_IV_MSD/t^(alpha)-const}
\Big|\frac{\E{X(t)^2}}{t^{\alpha}}- \frac{2\sin(\alpha\pi)}{\alpha\pi\beta C_\alpha
} \Big| = O(t^{-\eta/2}), \quad t \rightarrow \infty,
\end{align}
%\frac{-4\int_0^\infty\frac{\cos(z)}{z^\alpha}\d z \,\Gamma(-\alpha) \cos\Big(\frac{\alpha \pi}{2}\Big)}{\pi\beta C_\alpha\big[ \big(\int_0^\infty\frac{\cos(z)}{z^\alpha}\d z\big)^2+\big(\int_0^\infty\frac{\sin(z)}{z^\alpha}\d z\big)^2\big]}
where $C_\alpha=\lim_{t\to\infty}t^\alpha K(t)$, $\eta = \min\{\alpha,1-\alpha,\alpha\beta_\alpha\}$ and $\alpha,\,\beta_\alpha$ are constants from~\eqref{cond:K:convergent-rate:subdiffusion}.
\end{theorem}

The following theorem is the analog of Theorem \ref{thm:convergent-rate:diffusion+subdiffusion} in the critical regime. Similarly to Theorem~\ref{thm:convergent-rate:diffusion+subdiffusion}, the proof of Theorem~\ref{thm:convergent-rate:critical} draws upon an analysis of the small-frequency asymptotics of $\Kcos(\omega)$, i.e., as $\omega\to 0$.
\begin{theorem} \label{thm:convergent-rate:critical} Suppose that $K(t)$ satisfies~\eqref{cond1} and~~\eqref{cond:K:convergent-rate:critical}. Let $V$ be the weak solution of~\eqref{eq:gle} as in Definition~\ref{def:weak-solution} and let $X(t)$ be the position process associated with $V$ as in~\eqref{def:v(t)-process}. Then,
\begin{equation}\label{e:convergent-rate:critical}
\Big|\frac{\E{X(t)^2}}{ t/\log(t)}- \frac{2}{\beta C_1}\Big| = O(|\log(t)|^{-1}), \quad t \rightarrow \infty,
\end{equation}
where $C_1=\lim_{t\to\infty}tK(t)$ (see \eqref{e:K:convergent-rate:critical}).
\end{theorem}

\begin{remark}\label{r:extension_full_GLE}
Theorems \ref{thm:asymptotic:critical}, \ref{thm:convergent-rate:diffusion+subdiffusion} and \ref{thm:convergent-rate:critical} are only shown for the reduced family~\eqref{eq:gle}. However, extensions to the full equation~\eqref{eq:gle:full} can be established by similar arguments.
\end{remark}

%
%\begin{lemma} \label{lem:lim:log} Suppose that $g:(0,\infty)\to\rbb$ is a continuous function and that
%\begin{align}\label{lim:g(2t)}
%g(2 t)-g(t) \to 0,\qquad t\to\infty.
%\end{align}
%Then, \begin{align*}
%\frac{g(t)}{\log(t)}\to 0,\qquad t\to\infty.
%\end{align*}
%\end{lemma}
%\begin{proof} Fix $\epsilon>0$. In view of~\eqref{lim:g(2t)}, there exists $N>0$ sufficiently large such that for every $t\geq N$, it holds that
%\begin{align*}
%|f(2t)-f(t)|<\epsilon.
%\end{align*}
%For a given integer $k>0$ (to be chosen later), consider arbitrarily $t\in[2^{k}N,\infty)$. Then, there exists $k_1\geq k$ and $s\in[N,2N)$ such that $t=2^{k_1}s$. It follows from~\eqref{lim:g(2t)} that
%\begin{align*}
%|g(t)-g(s)|&= |g(2^{k_1}s)-g(s)|\\
%&\leq |g(2^{k_1}s)-g(2^{k_1-1}s)|+\dots+|g(2s)-g(s)|\\
%&\leq k_1\epsilon.
%\end{align*}
%As a consequence, we have
%\begin{align*}
%|g(t)|\leq |g(s)|+k_1\epsilon\leq \sup_{r\in[N,2N)}|g(r)|+k_1\epsilon.
%\end{align*}
%We divide through both sides by $\log(t)=\log(2^{k-1}s)$ to find that
%\begin{align*}
%\frac{|g(t)|}{\log(t)}&\leq \frac{\sup_{r\in[N,2N)}|g(r)|+k_1\epsilon}{\log(t)}\\
%&= \frac{\sup_{r\in[N,2N)}|g(r)|}{k_1\log(2)+\log(s)}+\frac{k_1\epsilon}{k_1\log(2)+\log(s)}\\
%&\leq \frac{\sup_{r\in[N,2N)}|g(r)|}{k\log(2)}+\frac{\epsilon}{\log(2)}.
%\end{align*}
%By taking $k$ sufficiently large, we obtain the result, thus completing the proof.
%\end{proof}

\section{Abelian-Tauberian Fourier analysis of memory} \label{sec:Fourier}
Throughout the rest of the paper, $c$ denotes a generic positive constant. The main parameters that it depends on will be indicated in parenthesis, e.g., $c(T,q)$ is a function of $T$ and $q$.

In this section, we introduce and establish the Fourier analysis results that are used in the subsequent sections. Recall that the usual Fourier transform of a function $\f\in L^1(\rbb)$ is given by
\begin{align*}
\widehat{\f}(\omega)=\int_\rbb e^{it\omega}\f(t)\d t, \quad \omega \in \bbR.
\end{align*}

First, we state the following lemma, which shows that $\Kcos$ and $\Ksin$ are well-defined under mild assumptions. For the sake of brevity, we omit its proof, which is similar to that of \cite[Lemma 2.18]{mckinley2018anomalous}. The estimate~\eqref{ineq:fcos-fsin} provided in the lemma is useful in establishing Fourier-type results on $\Kcos$ and $\Ksin$ (Propositions \ref{prop:asymptotic:Fcos-sin:critical}, \ref{prop:temper-distribution:Fcos} and Lemma \ref{lem:asymptotic:Fcos-sin:critical}).
\begin{lemma}\label{lem:fcos-fsin:well-defined} Suppose that $K$ satisfies~\eqref{cond1a} and \eqref{cond1b}. Then $\Kcos$ and $\Ksin$ are well-defined, continuous on $\omega\in(0,\infty)$ and converge to zero as $\omega\to\infty$. Furthermore, there exists a constant $A$ sufficiently large such that for every nonzero $\omega$ and $t\geq A$, %it holds that
\begin{align} \label{ineq:fcos-fsin}
\max\Big\{\Big|\int_t^\infty\close K(s)\cos(s\omega)\d s\Big|,\Big|\int_t^\infty\close K(s)\sin(s\omega)\d s\Big|\Big\}\leq \frac{4K(t)}{|\omega|}.
\end{align}
\end{lemma}
In Proposition~\ref{prop:asymptotic:Fcos-sin:critical}, stated and proved next, we provide an Abelian result for Fourier-type transforms when $K(t)\sim t^{-1}$ as $t\to\infty$. This proposition is, in turn, used in the proof of Theorem~\ref{thm:asymptotic:critical}, where we establish the large-time asymptotic growth of the MSD in the critical regime.

\begin{proposition}[Abelian direction]\label{prop:asymptotic:Fcos-sin:critical} Suppose that $K\in L^1_{\text{loc}}(0,\infty)$ satisfies conditions~\eqref{cond1b}and~ \eqref{cond:powerlaw:t}. Then,
\begin{equation}\label{e:lim_Ksin(omega)=const}
\lim_{\omega\to 0}\Ksin(\omega) = C_1 \hspace{1mm}\frac{\pi}{2},
\end{equation}
where $C_1=\lim_{t\to\infty}t\,K(t)$ (see \eqref{e:K:convergent-rate:critical}). Moreover,
\begin{equation}\label{e:Kcos-sim-logomega}
\Kcos(\omega)\sim |\log(\omega)|, \quad  \omega \to 0.
\end{equation}

\end{proposition}
\begin{proof}To show \eqref{e:lim_Ksin(omega)=const}, we first note that condition~\eqref{cond:powerlaw:t} implies that $t\,K(t)$ is bounded for $t\in[1,\infty)$. For $\omega>0$ small and $A$ large, we can re-express
\begin{align}\label{e:Ksin=I0+I1+I2}
\Ksin(\omega) = \int_0^\infty\close  K(t)\sin(\omega t)\d t &= \Big\{\int_0^1+\int_1^{A/\omega}\close +\int_{A/\omega}^\infty\Big\}  K(t)\sin(\omega t)\d t \nonumber\\% \close K(t)\sin(\omega t)\d t\\
&= I_0(\omega)+I_1(\omega)+I_2(\omega).
\end{align}
Since $K$ is locally integrable, the dominated convergence theorem readily implies that
\begin{align}\label{e:I0_for_Ksin}
I_0(\omega)\to 0,\qquad\omega\to 0.
\end{align}
In regard to $I_2(\omega)$, for $\omega > 0$ sufficiently small, $K(t)$ is decreasing for $t\in[A/\omega,\infty)$. Then, we can invoke~\eqref{ineq:fcos-fsin} to obtain
%\begin{align*}
%\int_{1/\omega}^A \close f(t)\sin(\omega t)\d t = f\Big(\frac{1}{\omega}\Big)\int_{1/\omega}^{A^*}\close  \sin(\omega t)\d t \leq f\Big(\frac{1}{\omega}\Big) \frac{2}{\omega}.
%\end{align*}
%It follows that
\begin{align}\label{e:I2_for_Ksin}
|I_2(\omega)|= \Big|\int_{A/\omega}^\infty\close K(t)\sin(\omega t)\d t \Big| \leq K\Big(\frac{A}{\omega}\Big)\frac{4}{\omega}\leq \frac{4}{A} \sup_{z \in [1,\infty)}z\,K(z).
\end{align}
Concerning $I_1(\omega)$, using a change of variable $z=t\omega$, we rewrite $I_1$ as
\begin{align*}
I_1 = \int_\omega^{A}\close K\Big(\frac{z}{\omega}\Big)\frac{\sin(z)}{\omega}\d z=\int_\omega^A\frac{z}{\omega}K\Big(\frac{z}{\omega}\Big)\frac{\sin(z)}{z}\d z.
\end{align*}
It follows from the dominated convergence theorem that
\begin{align}\label{e:I1_for_Ksin_bound}
I_1\to C_1\int_0^A\frac{\sin(z)}{z}\d z,\quad \omega\to 0.
\end{align}
Combining \eqref{e:Ksin=I0+I1+I2}--\eqref{e:I1_for_Ksin_bound} and \cite[p.\ 423, formula (3.721.1)]{gradshteyn:ryzhik:2007}, we obtain %the desired limit of $\Ksin$.
$$
\lim_{\omega\to 0}\Ksin(\omega) = C_1 \int_0^\infty \frac{\sin(z)}{z}\d z = C_1 \frac{\pi}{2}.
$$
This shows \eqref{e:lim_Ksin(omega)=const}.

Turning to \eqref{e:Kcos-sim-logomega}, note that
\begin{equation}\label{eqn:prop:asymptotic:Fcos-sin:critical:1}
\frac{\Kcos(\omega)}{|\log(\omega)|} =\frac{\Kcos(\omega)-\int_0^{1/\omega}K(t)\d t}{|\log(\omega)|}+\frac{\int_0^{1/\omega}K(t)\d t}{|\log(\omega)|}.
\end{equation}
However, by \cite[Theorem 7]{pitman1968behaviour},
\begin{align}
\Kcos(\omega)-\int_0^{1/\omega}\close\close K(t)\d t\to c<\infty, \quad \omega\to 0.
\end{align}
Therefore, the first fraction on the right-hand side of~\eqref{eqn:prop:asymptotic:Fcos-sin:critical:1} converges to zero as $\omega\to 0$. In regard to the second fraction, we can write
\begin{align}\label{e:lim_int_K/log-omega}
\lim_{\omega\to 0}\frac{\int_0^{1/\omega}K(t)\d t}{|\log(\omega)|}=\lim_{x\to\infty}\frac{\int_0^{x}K(t)\d t}{\log(x)}=\lim_{x\to\infty}xK(x)=C_1.
\end{align}
Expressions \eqref{eqn:prop:asymptotic:Fcos-sin:critical:1}--\eqref{e:lim_int_K/log-omega} imply \eqref{e:Kcos-sim-logomega}, as claimed.
\end{proof}

Under a mild additional assumption on the kernel function $K(t)$, a converse for expression \eqref{e:Kcos-sim-logomega} in Proposition~\ref{prop:asymptotic:Fcos-sin:critical} can be established that is of interest in its own right. To be precise, we have the following Tauberian-type proposition. %Proposition~\ref{prop:asymptotic:Fcos-sin:critical:Tauberian}, whose proof is relatively short. Although we will not employ this proposition throughout the paper, we include it here for the sake of completeness.
\begin{proposition}[Tauberian direction] \label{prop:asymptotic:Fcos-sin:critical:Tauberian} Suppose $K \in L^1_{\text{loc}}(0,\infty)$ satisfies~\eqref{cond1b}, and that
\begin{equation}\label{e:suptK(t)<infty}
\sup_{t \in [1,\infty)}|t\,K(t)|<\infty.
\end{equation}
If $\Kcos(\omega)\sim |\log(\omega)|$ as $\omega\to 0^+$, then
\begin{align}\label{e:K(t)_decays_t^(-1)}
K(t)\sim t^{-1}, \quad t\to\infty.
\end{align}
\end{proposition}
\begin{remark} It can be shown that $K(t)\sim t^{-1}$ as $t\to\infty$ if and only if for every $\lambda>1$, $\Kcos(\lambda \omega)-\Kcos(\omega)\to \log(\lambda)$ as $\omega\to 0$ \cite{inoue1995abel}. However, this statement should not be confused with those of Propositions~\ref{prop:asymptotic:Fcos-sin:critical} and \ref{prop:asymptotic:Fcos-sin:critical:Tauberian}.
\end{remark}
In order to prove Proposition~\ref{prop:asymptotic:Fcos-sin:critical:Tauberian}, we need the following Lemma.

\begin{lemma}\label{lem:asymptotic:Fcos-sin:critical} Suppose $K(t)$ satisfies the conditions of Proposition \ref{prop:asymptotic:Fcos-sin:critical:Tauberian}. %~\eqref{cond1b} and that $\sup_{[1,\infty)}t\,K(t)<\infty$.
Then,
\begin{align}\label{e:lim_Kcos-integ/|log|}
\lim_{\omega\to 0^+} \frac{\Kcos(\omega)-\int_0^{1/\omega}K(t)\d t}{|\log(\omega)|}=0.
\end{align}
\end{lemma}
\begin{proof}
Fix an arbitrary $\epsilon>0$. We can write
\begin{equation} \label{eqn:lem:asymptotic:Fcos-sin:critical:1}
\begin{aligned}
\Kcos(\omega)-\int_0^{1/\omega}\close K(t)\d t &= \Big\{\int_0^{1}\close+\!\int_1^{1/\omega}\Big\} K(t)(\cos(t\omega)-1)\d t\\
&\qquad+ \Big\{\int_{1/\omega}^{1/\omega^{1+\epsilon}}\close\close+\int_{1/\omega^{1+\epsilon}}^\infty\Big\}  K(t)\cos(t\omega)\d t.
\end{aligned}
\end{equation}
Concerning the first two integrals on the right-hand side of \eqref{eqn:lem:asymptotic:Fcos-sin:critical:1}, without loss of generality, suppose $0 < \omega < 1$. Then,
\begin{equation}\label{ineq:lem:asymptotic:Fcos-sin:critical:2}
\begin{aligned}
\Big| \Big\{\int_0^{1}\close+\!\int_1^{1/\omega}\Big\} K(t)(\cos(t\omega)-1)\d t \Big|  &\leq 2\int_0^1\close K(t)\d t+\int_1^{1/\omega}\close\close K(t)\,t\,\omega\d t\\
&\leq 2\int_0^1\close K(t)\d t + c\, \omega \Big(\frac{1}{\omega}-1\Big)\sup_{t \in [1,\infty)}t\, K(t)\\
&\leq 2\int_0^1\close K(t)\d t + c.
\end{aligned}
\end{equation}
where the last inequality follows from condition \eqref{e:suptK(t)<infty}. % in the first implication, we have used the inequality $1-\cos(x)\leq  x $ for every $x\geq 0$, and in the last implication, we have employed the fact that $t\,K(t)$ is bounded on $[1,\infty)$.
Likewise, with regards to the third integral on the right-hand side of~\eqref{eqn:lem:asymptotic:Fcos-sin:critical:1}, %we estimate
\begin{equation}\label{ineq:lem:asymptotic:Fcos-sin:critical:3}
\begin{aligned}
\Big|\int_{1/\omega}^{1/\omega^{1+\epsilon}}\close\close\close K(t)\cos(t\omega)\d t \Big| = \Big|\int_{1/\omega}^{1/\omega^{1+\epsilon}}\close\close \close t\,K(t)\frac{\cos(t\omega)}{t}\d t \Big| \leq c\int_{1/\omega}^{1/\omega^{1+\epsilon}}\close\close\close t^{-1}\d t= c\,\epsilon|\log(\omega)|.
\end{aligned}
\end{equation}
Concerning the last integral on the right-hand side of~\eqref{eqn:lem:asymptotic:Fcos-sin:critical:1}, we note that for $\omega > 0$ sufficiently small, $K(t)$ is decreasing on $[1/\omega^{1+\epsilon},\infty)$. By \eqref{ineq:fcos-fsin},
\begin{equation}\label{ineq:lem:asymptotic:Fcos-sin:critical:4}
\begin{aligned}
\int_{1/\omega^{1+\epsilon}}^\infty \close\close K(t)\cos(t\omega)\d t\leq \frac{4}{\omega}K\Big(\frac{1}{\omega^{1+\epsilon}}\Big)=4\omega^{\epsilon}\frac{1}{\omega^{1+\epsilon}}K\Big(\frac{1}{\omega^{1+\epsilon}}\Big)\leq c\,\omega^{\epsilon},
\end{aligned}
\end{equation}
where the last inequality is a consequence of condition \eqref{e:suptK(t)<infty}. Expressions \eqref{eqn:lem:asymptotic:Fcos-sin:critical:1}--\eqref{ineq:lem:asymptotic:Fcos-sin:critical:4} imply that
\begin{align*}
\frac{|\Kcos(\omega)-\int_0^{1/\omega}K(t)\d t|}{|\log(\omega)|}\leq \frac{c+c\,\epsilon|\log(\omega)|+c\omega^{\epsilon}}{|\log(\omega)|},
\end{align*}
whence
\begin{align*}
\limsup_{\omega\to 0^+} \frac{\Kcos(\omega)-\int_0^{1/\omega}K(t)\d t}{|\log(\omega)|}\leq c\,\epsilon,
\end{align*}
where the constant $c>0$ is independent of $\epsilon$. Since $\epsilon > 0$ is arbitrary, \eqref{e:lim_Kcos-integ/|log|} holds. %The resulting limit is thus obtained by shrinking $\epsilon$ further to zero.
\end{proof}
With Lemma~\ref{lem:asymptotic:Fcos-sin:critical} in hand, the proof of Proposition~\ref{prop:asymptotic:Fcos-sin:critical:Tauberian}, provided next, is relatively short.
\begin{proof}[Proof of Proposition~\ref{prop:asymptotic:Fcos-sin:critical:Tauberian}] %Since $\sup_{t \in [1,\infty)}t\, K(t)$ is finite, Lemma~\ref{lem:asymptotic:Fcos-sin:critical} implies that
%\begin{align*}
%\frac{\Kcos(\omega)-\int_0^{1/\omega}K(t)\d t}{|\log(\omega)|}\to 0,\qquad\omega\to 0^+.
%\end{align*}
Consider the decomposition~\eqref{eqn:prop:asymptotic:Fcos-sin:critical:1}. By Lemma~\ref{lem:asymptotic:Fcos-sin:critical}, the first quotient on the right-hand side vanishes as $\omega \rightarrow 0^+$. It follows that %Recalling from~\eqref{eqn:prop:asymptotic:Fcos-sin:critical:1}, we readily see that
\begin{align*}
\lim_{\omega\to 0^+}\frac{\int_0^{1/\omega}K(t)\d t}{|\log(\omega)|}=\lim_{\omega\to 0}\frac{\Kcos(\omega)}{|\log(\omega)|}=C_1>0.
\end{align*}
By the same reasoning as in \eqref{e:lim_int_K/log-omega}, $C_1=\lim_{x\to\infty}xK(x)$, which shows \eqref{e:K(t)_decays_t^(-1)}.
%which completes the proof.
\end{proof}

While Proposition~\ref{prop:asymptotic:Fcos-sin:critical} is sufficient for determining the large-time asymptotic growth of the MSD in the critical regime, it does not provide information on the convergence rate of the Fourier-type transforms \eqref{e:Kcos_Ksin} near the origin.
We will see later in the proof of Theorem~\ref{thm:convergent-rate:diffusion+subdiffusion} and \ref{thm:convergent-rate:critical} that this information is crucial in establishing the growth rate of the MSD in all regimes.

We now state and show three auxiliary results (Lemmas~\ref{lem:Kcos:convergent-rate:diffusion}, \ref{lem:Kcos:convergent-rate:subdiffusion} and \ref{lem:Kcos:convergent-rate:critical}) that are used in Section~\ref{sec:MSD:convergent-rate} to establish the convergence rate of the MSD towards its limit in each regime. We start off with the diffusive regime.

\begin{lemma}[Diffusive regime] \label{lem:Kcos:convergent-rate:diffusion} Suppose that $K$ satisfies conditions \eqref{cond1a}, \eqref{cond1b} and~\eqref{cond:K:convergent-rate:diffusion}. Then, for constants $c_1,c_2 > 0$ and $\omega \in \bbR$,
\begin{align}\label{e:Kcos-Kcos0}
|\Kcos(\omega)-\Kcos(0)|\leq c_1 \hspace{1mm}\omega^{\gamma_{0}}
\end{align}
and
\begin{align}\label{e:|Ksin|=O(omega^gamma01)}
|\Ksin(\omega)| \leq c_2 \hspace{1mm}\omega^{\gamma_{0,1}},
\end{align}
where $\gamma_{0}=\min\{\beta_0,2\}$, $\gamma_{0,1}=\min\{\beta_0,1\}$  and $\beta_0$ is the exponent constant from~\eqref{cond:K:convergent-rate:diffusion}.
\end{lemma}
\begin{remark} The bounds $\gamma_0 \leq 2$ and $\gamma_{0,1}\leq 1$ in Lemma \ref{lem:Kcos:convergent-rate:diffusion} cannot be improved regardless of how large $\beta_0$ is. This seems to be the simplest formulation. To see this, consider the memory kernel instance $K(t)=e^{-|t|}$. Then, $t^{\beta_0}K(t)$ is integrable for every $\beta_0>0$. Moreover, its Fourier-type transforms are given by
\begin{align*}
\Kcos(\omega)=\frac{1}{1+\omega^2},\quad\text{and}\quad \Ksin(\omega)=\frac{\omega}{1+\omega^2}.
\end{align*}
It is straightforward to verify that, for the above $K$, $\gamma_0=2$ and $\gamma_{0,1}=1$.
\end{remark}

\begin{proof}[Proof of Lemma \ref{lem:Kcos:convergent-rate:diffusion} ] %To prove the estimate on $\Ksin$, we will employ the inequality $\sin(x)\leq x^{\gamma_{0,1}}$ for every $x\geq 0$ as follows.

We first show \eqref{e:|Ksin|=O(omega^gamma01)}. In fact, by applying the elementary bound $|\sin(x)|\leq x^{\gamma_{0,1}}$, $x\geq 0$ and condition~\eqref{cond:K:convergent-rate:diffusion},
\begin{align*}
|\Ksin(\omega)| = \Big|\int_0^\infty\close K(t)\sin(t\omega)\d t \Big| &\leq \omega \int_0^1\close t\,K(t)\d t+\omega^{\gamma_{0,1}}\int_1^\infty \close t^{\gamma_{0}}K(t)\d t\\
&\leq  \omega \int_0^1\close K(t)\d t+\omega^{\gamma_{0,1}}\int_1^\infty \close t^{\beta_0}K(t)\d t\\
&=O(\omega^{\gamma_{0,1}}).
\end{align*}
%since $t^{\beta_0}K(t)$ is integrable by condition~\eqref{cond:K:convergent-rate:diffusion}.

Next, we prove \eqref{e:Kcos-Kcos0}. For every $\omega > 0$,
\begin{align*}
|\Kcos(\omega)-\Kcos(0)|= \Big| \int_0^\infty\close K(t)(1-\cos(t\omega))\d t \Big|
\leq c\int_0^\infty\close K(t)t^{\gamma_{0}}\omega^{\gamma_{0}}\d t,
\end{align*}
where we use the inequality $1-\cos(x)\leq c |x|^{\gamma_{0}}$ for any $\gamma_{0}\in[0,2]$. It follows that
\begin{align*}
|\Kcos(\omega)-\Kcos(0)| \leq c\,\omega^{\gamma_{0}}\int_0^\infty\close K(t)t^{\beta_0}\d t,
\end{align*}
which implies \eqref{e:Kcos-Kcos0}.
\end{proof}

In regard to the convergence rate of the Fourier transforms in the subdiffusive regime, we have the following lemma.
\begin{lemma}[Subdiffusive regime] \label{lem:Kcos:convergent-rate:subdiffusion} Suppose that $K$ satisfies conditions \eqref{cond1a},~\eqref{cond1b} and~\eqref{cond:K:convergent-rate:subdiffusion}. Then, as $\omega\to 0$,
\begin{align}\label{e:w^(1-alpha)*Kcos-integ}
\Big|\omega^{1-\alpha}\Kcos(\omega)-C_\alpha\int_0^\infty \frac{\cos(z)}{z^\alpha}\d z\Big|= O(\omega^{\gamma_{\alpha}})
\end{align}
and
\begin{align}\label{e:w^(1-alpha)*Ksin-integ}
\Big|\omega^{1-\alpha}\Ksin(\omega)-C_\alpha\int_0^\infty \frac{\sin(z)}{z^\alpha}\d z\Big|= O(\omega^{\gamma_{\alpha}}),
\end{align}
where $C_\alpha=\lim_{t\to \infty}t^{\alpha}K(t)$, $\gamma_{\alpha}=\min\{1-\alpha,\alpha\beta_{\alpha}\}$ and $\alpha,\,\beta_{\alpha}$ are the exponent constants from~\eqref{cond:K:convergent-rate:subdiffusion}.
\end{lemma}
%\GDcomment{I feel the formulas for $\int_0^\infty\frac{\cos(z)}{z^\alpha}\d z$ and $\int_0^\infty\frac{\sin(z)}{z^\alpha}\d z$ provided in \cite{gradshteyn:ryzhik:2007} are not very enlightening, so let's just stick with the integrals.}

\begin{proof} We only need to prove \eqref{e:w^(1-alpha)*Kcos-integ}: claim \eqref{e:w^(1-alpha)*Ksin-integ} can be shown simply by replacing cosines with sines throughout the argument.

Let $\delta>0$ be a constant that will be chosen later. For $\omega\in(0,1)$, recast
\begin{align}\label{e:w(1-alpha)intKcosdt}
\omega^{1-\alpha}\int_0^\infty\close K(t)\cos(t\omega)\d t  = \omega^{1-\alpha}\Big\{\int_0^1+\int_1^{\omega^{-\delta-1}}\close\close\close+\int_{\omega^{-\delta-1}}^\infty  \Big\} K(t)\cos(t\omega)\d t.
\end{align}
We now proceed to reexpress or construct bounds, in absolute value, for each integral term on the right-hand side of \eqref{e:w(1-alpha)intKcosdt}. In regard to the first term in \eqref{e:w(1-alpha)intKcosdt},
\begin{align}\label{e:w(1-alpha)intKcosdt_1st_integral}
\omega^{1-\alpha}\Big|\int_0^1\close K(t)\cos(t\omega)\d t \Big| \leq \omega^{1-\alpha}\int_0^1\close |K(t)|\d t =O(\omega^{1-\alpha}).
\end{align}
As for the third term in \eqref{e:w(1-alpha)intKcosdt}, assuming $\omega$ is sufficiently small, Lemma~\ref{lem:fcos-fsin:well-defined} implies that
\begin{align}\label{e:w(1-alpha)intKcosdt_3rd_integral}
\omega^{1-\alpha} \Big|\int_{\omega^{-\delta-1}}^\infty\close\close K(t)\cos(t\omega)\d t\Big| &\leq c\,\omega^{1-\alpha}\frac{|K(\omega^{-\delta-1})|}{\omega}  \nonumber  \\
&= c\,\omega^{-\alpha}|K(\omega^{-\delta-1})|\omega^{-(\delta+1)\alpha}\omega^{(\delta+1)\alpha} \nonumber \\
&=O(\omega^{\alpha\delta}).
\end{align}
In \eqref{e:w(1-alpha)intKcosdt_3rd_integral}, the last equality is a consequence of the fact that $t^\alpha K(t)$ is bounded as $t\to\infty$. Moreover, by a change of variable $z=t\omega$, the middle (second) integral term in \eqref{e:w(1-alpha)intKcosdt} can be rewritten as
\begin{align}\label{e:w(1-alpha)intKcosdt_2nd_integral}
\omega^{1-\alpha}\int_1^{\omega^{-\delta-1}}\close\close\close K(t)\cos(t\omega)\d t=\int_\omega^{\omega^{-\delta}}\close\Big(\frac{z}{\omega}\Big)^\alpha K\Big(\frac{z}{\omega}\Big)\frac{\cos(z)}{z^\alpha}\d z.
\end{align}
Expressions \eqref{e:w(1-alpha)intKcosdt_1st_integral}, \eqref{e:w(1-alpha)intKcosdt_3rd_integral} and \eqref{e:w(1-alpha)intKcosdt_2nd_integral} imply that
\begin{align}\label{e:w(1-alpha)intKcosdt_reexpressed}
\omega^{1-\alpha}\int_0^\infty\close K(t)\cos(t\omega)\d t  = O(\omega^{1-\alpha})+\int_\omega^{\omega^{-\delta}}\close\Big(\frac{z}{\omega}\Big)^\alpha K\Big(\frac{z}{\omega}\Big)\frac{\cos(z)}{z^\alpha}\d z+O(\omega^{\alpha\delta}).
\end{align}
Likewise,
\begin{align}\label{e:int_cos(z)/z^alpha dz_reexpressed}
C_\alpha\int_0^\infty \frac{\cos(z)}{z^\alpha}\d z & =  C_\alpha \Big\{\int_0^\omega+\int_\omega^{\omega^{-\delta}}\close\close+\int_{\omega^{-\delta}}^\infty \Big\} \frac{\cos(z)}{z^\alpha}\d z \nonumber \\
&=O(\omega^{1-\alpha})+C_\alpha\int_\omega^{\omega^{-\delta}}\frac{\cos(z)}{z^\alpha}\d z+O(\omega^{\alpha\delta}).
\end{align}
By \eqref{e:w(1-alpha)intKcosdt_reexpressed} and \eqref{e:int_cos(z)/z^alpha dz_reexpressed}, % Collecting everything now yields
\begin{multline}\label{e:w(1-alpha)intKcosdt_- int_cos(z)/z^alpha dz_reexpressed}
\omega^{1-\alpha}\int_0^\infty\close K(t)\cos(t\omega)\d t-C_\alpha\int_0^\infty \frac{\cos(z)}{z^\alpha}\d z\\ =O(\omega^{1-\alpha})+O(\omega^{\alpha\delta})+\int_\omega^{\omega^{-\delta}}\Big[\Big(\frac{z}{\omega}\Big)^\alpha K\Big(\frac{z}{\omega}\Big)-C_\alpha\Big]\frac{\cos(z)}{z^\alpha}\d z.
\end{multline}
In regard to the integral term on the right-hand side of \eqref{e:w(1-alpha)intKcosdt_- int_cos(z)/z^alpha dz_reexpressed}, we invoke~\eqref{cond:K:convergent-rate:subdiffusion} to arrive at the bound
\begin{align}\label{e:w^(1-alpha)_int_Kcosdt-Cint_cos(z)/z^alphadz}
\Big|\int_\omega^{\omega^{-\delta}}\Big[\Big(\frac{z}{\omega}\Big)^\alpha K\Big(\frac{z}{\omega}\Big)-C_\alpha\Big]\frac{\cos(z)}{z^\alpha}\d z \Big| \leq c\,\omega^{\beta_{\alpha}}\int_\omega^{\omega^{-\delta}}\close\close \frac{1}{z^{\alpha+\beta_{\alpha}}}\d z.
\end{align}
Turning back to expression \eqref{e:w(1-alpha)intKcosdt}, set $\delta=\beta_{\alpha}$. There are two cases pertaining to the sum $\alpha+\beta_{\alpha}$ in the bound \eqref{e:w^(1-alpha)_int_Kcosdt-Cint_cos(z)/z^alphadz}. First, if $\alpha+\beta_{\alpha}=1$, then
\begin{align}\label{e:omega^(betaalpha)int_1/z^(alpha+betaalpha)_dz_1st_bound}
c\,\omega^{\beta_{\alpha}}\int_\omega^{\omega^{-\delta}}\close\close \frac{1}{z^{\alpha+\beta_{\alpha}}}\d z = c\omega^{\beta_{\alpha}}|\log(\omega)|\leq c\, \omega^{\alpha\beta_{\alpha}}.
\end{align}
Otherwise, i.e., if $\alpha+\beta_{\alpha} \neq 1$, then
\begin{align}\label{e:omega^(betaalpha)int_1/z^(alpha+betaalpha)_dz_2nd_bound}
c\,\omega^{\beta_{\alpha}}\int_\omega^{\omega^{-\delta}}\close\close \frac{1}{z^{\alpha+\beta_{\alpha}}}\d z \leq  c(\omega^{\beta_{\alpha}-\delta(1-\alpha-\beta_{\alpha})}+\omega^{1-\alpha})&=c(\omega^{\beta_{\alpha}-\beta_{\alpha}(1-\alpha-\beta_{\alpha})}+\omega^{1-\alpha}) \nonumber \\
&=O(\omega^{\alpha\beta_{\alpha}}+\omega^{1-\alpha}).
\end{align}
Therefore, by expressions \eqref{e:w(1-alpha)intKcosdt_- int_cos(z)/z^alpha dz_reexpressed}--\eqref{e:omega^(betaalpha)int_1/z^(alpha+betaalpha)_dz_2nd_bound}, %Finally, we collect everything to arrive at
\begin{align*}
\Big|\omega^{1-\alpha}\Kcos(\omega)-C_\alpha\int_0^\infty \frac{\cos(z)}{z^\alpha}\d z \Big| = O(\omega^{\alpha\beta_{\alpha}}+\omega^{1-\alpha}).
\end{align*}
This establishes \eqref{e:w^(1-alpha)*Kcos-integ}.
\end{proof}

Concerning the critical regime, we have the following result.
\begin{lemma}[Critical regime] \label{lem:Kcos:convergent-rate:critical} Suppose that $K$ satisfies conditions \eqref{cond1a},~\eqref{cond1b} and~\eqref{cond:K:convergent-rate:critical}. Then,
\begin{align*}
\Big|\frac{\Kcos(\omega)}{|\log(\omega)|}-C_1 \Big|= O(|\log(\omega)|^{-1}),\quad \omega\to 0^+,
\end{align*}
where $C_1=\lim_{t\to \infty}tK(t)$ (see \eqref{e:K:convergent-rate:critical}).
\end{lemma}
\begin{proof}
Recast
\begin{align}\label{eqn:lem:Kcos:convergent-rate:critical:1}
\frac{\Kcos(\omega)}{|\log(\omega)|}-C_1=\frac{\Kcos(\omega)-\int_0^{1/\omega}K(t)\d t}{|\log(\omega)|}+\frac{\int_0^{1/\omega}K(t)\d t}{|\log(\omega)|}-C_1.
\end{align}
To construct a bound for the first ratio on the right-hand side of \eqref{eqn:lem:Kcos:convergent-rate:critical:1}, we shall improve upon the proof of Lemma~\ref{lem:asymptotic:Fcos-sin:critical}. To be precise, we sharpen the estimate~\eqref{ineq:lem:asymptotic:Fcos-sin:critical:3} by making the change of variable $z=t\omega$, i.e.,
\begin{equation}
\begin{aligned}\label{ineq:lem:asymptotic:Fcos-sin:critical:3-sharper}
\int_{1/\omega}^{1/\omega^{1+\epsilon}}\close\close\close K(t)\cos(t\omega)\d t &= \int_{1}^{1/\omega^{\epsilon}}\close \Big(\frac{z}{\omega}\Big)K\Big(\frac{z}{\omega}\Big)\frac{\cos(z)}{z}\d z\\
&= \int_{1}^{1/\omega^{\epsilon}}\close \Big[\Big(\frac{z}{\omega}\Big)K\Big(\frac{z}{\omega}\Big)-C_1\Big]\frac{\cos(z)}{z}\d z+C_1\int_{1}^{1/\omega^{\epsilon}}\frac{\cos(z)}{z}\d z.
\end{aligned}
\end{equation}
It is clear that the second integral on the right-hand side of \eqref{ineq:lem:asymptotic:Fcos-sin:critical:3-sharper} converges to $C_1\int_1^\infty\frac{\cos(z)}{z}\d z$ as $\omega\to 0$. Concerning the first integral, we invoke~\eqref{cond:K:convergent-rate:critical} to arrive at
\begin{align*}
\Big|\int_{1}^{1/\omega^{\epsilon}}\close \Big[\Big(\frac{z}{\omega}\Big)K\Big(\frac{z}{\omega}\Big)-C_1\Big]\frac{\cos(z)}{z}\d z \Big|
\leq \omega^{\beta_1}\int_{1}^{1/\omega^{\epsilon}}\frac{|\cos(z)|}{z^{1+\beta_1}}\d z \leq c\,\omega^{\beta_1},
\end{align*}
whence
\begin{align}\label{ineq:lem:Kcos:convergent-rate:critical:2}
\Big|\int_{1/\omega}^{1/\omega^{1+\epsilon}}\close K(t)\cos(t\omega)\d t  \Big| \leq c(\omega^{\beta_1}+1).
\end{align}
Combining~\eqref{ineq:lem:Kcos:convergent-rate:critical:2}, \eqref{eqn:lem:asymptotic:Fcos-sin:critical:1}, \eqref{ineq:lem:asymptotic:Fcos-sin:critical:2} and \eqref{ineq:lem:asymptotic:Fcos-sin:critical:4} yields the estimate
\begin{align}\label{ineq:lem:Kcos:convergent-rate:critical:3}
\Big|\frac{\Kcos(\omega)-\int_0^{1/\omega}K(t)\d t}{|\log(\omega)|} \Big|\leq \frac{c+c(\omega^{\beta_1}+1)+c\omega^{\epsilon}}{|\log(\omega)|}=O(|\log(\omega)|^{-1}).
\end{align}
With regards to the second term on the right-hand side of~\eqref{eqn:lem:Kcos:convergent-rate:critical:1}, it is straightforward to see that
\begin{equation}\label{ineq:lem:Kcos:convergent-rate:critical:4}
\begin{aligned}
\Big|\frac{\int_0^{1/\omega}K(t)\d t}{|\log(\omega)|}-C_1 \Big| &= \Big|\frac{1}{|\log(\omega)|}\int_0^{1}\close K(t)\d t+\frac{1}{|\log(\omega)|}\int_1^{1/\omega}\frac{t\,K(t)-C_1}{t}\d t \Big| \\
&\leq \frac{1}{|\log(\omega)|}\int_0^{1}\close K(t)\d t+\frac{c}{|\log(\omega)|}\int_1^{1/\omega}\close\frac{1}{t^{1+\beta_1}}\d t\\
&=O(|\log(\omega)|^{-1}).
\end{aligned}
\end{equation}
The result now follows immediately from~\eqref{ineq:lem:Kcos:convergent-rate:critical:3} and~\eqref{ineq:lem:Kcos:convergent-rate:critical:4}. The proof is thus complete.
\end{proof}

Let $\Sc$ be the Schwartz space of all smooth functions whose derivatives are rapidly decreasing. Recall that its dual space $\Sc'$ is the so-named class of tempered distributions on $\Sc$. For a given tempered distribution $g\in\Sc'$, $\Fcal[g]\in\Sc'$ denotes the Fourier transform of $g$ in $\Sc'$. It is well known that this transformation is a one-to-one relation in $\Sc'$. We conclude this section with a proposition on the Fourier transform of $K$, in the sense of tempered distributions, in the critical regime. We make use of Proposition~\ref{prop:temper-distribution:Fcos} later in Section~\ref{sec:wellposed} for the analysis on the well-posedness of~\eqref{eq:gle}.

\begin{proposition}\label{prop:temper-distribution:Fcos} Suppose that $K$ satisfies~\eqref{cond1a}, \eqref{cond1b} and \eqref{cond:powerlaw:t}. Then, $2\Kcos$ is the Fourier transform of $K$ in the sense of tempered distributions, i.e., for every $\f\in\Sc$, %it holds that
\begin{align}\label{e:temper-distribution:Fcos}
\int_\rbb K(t)\widehat{\f}(t)\d t = \int_\rbb 2\Kcos(\omega)\f(\omega)\d\omega.
\end{align}
\end{proposition}
\begin{proof} Since $K$ satisfies~\eqref{cond:powerlaw:t}, then $\Kcos(\omega)\sim |\log(\omega)|$ as $\omega\to 0$ by virtue of Proposition~\ref{prop:asymptotic:Fcos-sin:critical}. It follows that $\Kcos$ is integrable about the origin. Also, by Lemma~\ref{lem:fcos-fsin:well-defined}, it is continuous and converges to zero as $\omega\to\infty$. Thus, for every function $\f \in {\mathcal S}$,
\begin{align*}
\int_\rbb |\Kcos(\omega)\f(\omega)|\d\omega<\infty.
\end{align*}
We now consider a truncation of $K$ by setting $K_n(t)=K(t)1_{[-n,n]}(t)$. Since $K_n$ is integrable and symmetric, then
\begin{align}\label{e:integ_Kn*phi-hat}
\int_\rbb K_n(t)\widehat{\f}(t)\d t = \int_\rbb 2\Kcos^n(\omega)\f(\omega)\d\omega,
\end{align}
where $\Kcos^n(\omega):=\int_0^\infty K_n(t)\cos(t\omega)\d t = \int_0^n K(t)\cos(t\omega)\d t$. As $n\to\infty$, the integral on the left-hand side of \eqref{e:integ_Kn*phi-hat} converges to $\int_\rbb K(t)\widehat{\f}(t)\d t$. To establish \eqref{e:temper-distribution:Fcos}, it remains to show that
\begin{align} \label{lim:prop:temper-distribution:Fcos:1}
\int_\rbb \Kcos^n(\omega)\f(\omega)\d\omega\to\int_\rbb \Kcos(\omega)\f(\omega)\d\omega,\qquad n\to\infty.
\end{align}
To this end, note that, by Lemma~\ref{lem:fcos-fsin:well-defined}, $\Kcos$ is well-defined in the sense of improper Riemann integration. It follows that, for any $\omega\neq 0$, we have
\begin{align*}
\Kcos^n(\omega) = \int_0^n \close K(t)\cos(t\omega)\d t\to\int_0^\infty\close K(t)\cos(t\omega)\d t=\Kcos(\omega),\quad n\to\infty.
\end{align*}
On one hand, for every $|\omega|>1/n$, we have
\begin{align*}
|\Kcos^n(\omega)|\leq |\Kcos^n(\omega)-\Kcos(\omega)|+|\Kcos(\omega)|.
\end{align*}
For $n$ sufficiently large, inequality~\eqref{ineq:fcos-fsin} implies that
\begin{align*}
|\Kcos^n(\omega)-\Kcos(\omega)|&=\Big|\int_n^\infty\close K(t)\cos(t\omega)\d t \Big|\leq \frac{4K(n)}{|\omega|}\leq 4nK(n)< C,
\end{align*}
since $K(t)\sim 1/t$ as $t\to\infty$. Thus, when $|\omega|>1/n$, %it holds that
\begin{align*}
1_{\{|\omega|>1/n\}}(\omega)|\Kcos^n(\omega)|\leq |\Kcos(\omega)|+C.
\end{align*}
As a consequence of the dominated convergence theorem,
\begin{align}\label{e:int_|omega|>1/n}
\int_\rbb 1_{\{|\omega|>1/n\}}(\omega)\Kcos^n(\omega)\f(\omega)\d\omega\to\int_\rbb \Kcos(\omega)\f(\omega)\d\omega,\qquad n\to\infty.
\end{align}
On the other hand, we have that
\begin{align}\label{e:int_|omega|=<1/n}
\int_{|\omega|<1/n}\close \close|\Kcos^n(\omega)\f(\omega)|\d\omega &=\int_{|\omega|<1/n}  \Big|\int_0^n\close K(t)\cos(t\omega)\d t \Big| \hspace{2mm} |\f(\omega)|\d\omega \nonumber \\
%&\leq \int_0^n\close K(t)\d t\sup_{\omega \in \rbb}|\f(\omega)|\int_{|\omega|<1/n}\close\close1\d\omega\\
&\leq \frac{2\sup_{\omega \in \rbb}|\f(\omega)|}{n}\Big[\int_0^1\close K(t)\d t+\int_1^n \close K(t)\d t\Big] \nonumber  \\
&\leq \frac{2\sup_{\omega \in \rbb}|\f(\omega)|}{n}\Big[\int_0^1\close K(t)\d t+c\int_1^n  \frac{1}{t}\d t\Big] \nonumber \\
&= \frac{2\sup_{\omega \in \rbb}|\f(\omega)|}{n}\Big[\int_0^1\close K(t)\d t+c\log(n)\Big]\to 0,
\end{align}
as $n \rightarrow \infty$. Relations \eqref{e:int_|omega|>1/n} and \eqref{e:int_|omega|=<1/n} imply~\eqref{lim:prop:temper-distribution:Fcos:1}, which completes the proof.
\end{proof}

\section{Well-posedness and regularity} \label{sec:wellposed}
We now briefly review the framework of stationary solutions of~\eqref{eq:gle} introduced in~\cite{mckinley2018anomalous}. Let $\nu$ be a non-negative measure on $\rbb$ satisfying the condition
\begin{equation} \label{ineq:spectral-measure}
\int_\rbb \frac{\nu(\d x)}{(1+x^2)^k}<\infty
\end{equation}
for some integer $k$. Also, let $L^2(\Omega)$ be the space of squared integrable complex-valued Gaussian random variables. It is well known that $\nu$ is characterized by some $g \in \Sc' $ -- i.e., a tempered distribution -- and a stationary random distribution
\begin{equation} \label{e:F:S->L2}
F:\Sc\to L^2(\Omega)
\end{equation}
in the sense that, for $\f_1,\,\f_2\in\Sc$, %it holds that
\begin{align} \label{eq:stationary-distribution}
\E{\langle F,\f_1\rangle\overline{\langle F,\f_2\rangle}}= \langle g,\f_1*\widetilde{\f_2}\rangle = \int_\rbb \widehat{\f_1}(\omega)\overline{\widehat{\f_2}(\omega)}\nu(\d\omega),
\end{align}
where $\widetilde{f}(x):=f(-x)$~\cite{ito1954stationary}. In \eqref{eq:stationary-distribution}, $\langle F,\f\rangle$ and $\langle g,\f\rangle$ denote the so-named actions of $F$ and $g$ on $\f\in\Sc$, respectively. Moreover, $g$ is called the \emph{covariance distribution} and $\nu$ is called the \emph{spectral measure} of $F$. If $\nu$ is absolutely continuous with respect to Lebesgue measure, then we can extend $F$ in \eqref{e:F:S->L2} to an operator
\begin{equation} \label{e:V:S'->L^2}
V:\Sc'\to L^2(\Omega)
\end{equation}
such that, for $g_1,\,g_2\in\Sc'$ \cite{mckinley2018anomalous},
\begin{equation} \label{eqn:V.op}
\E{\langle V,g_1\rangle\overline{\langle V,g_2\rangle}}=\int_\rbb \Fcal[g_1](\omega)\overline{\Fcal[g_2](\omega)}\nu(\d\omega).
\end{equation}
The domain of $V$, denoted by
\begin{equation}\label{e:Dom(V)}
\text{Dom}(V),
\end{equation}
consists of those $g \in \Sc$ such that $\Fcal[g]$ is a complex-valued function and $\Fcal[g] \in L^2(\nu)$, the Hilbert space of $\nu$-squared integrable functions. It is worthwhile noting that, for a generic tempered distribution $g$, $\Fcal[g]$ is also a tempered distribution, which may not be a function. However, in order for $g$ to be included in Dom$(V)$, $\Fcal[g]$ has to be a complex-valued function.

Based on the operator $V$ as in \eqref{e:V:S'->L^2}, we can define the velocity and displacement processes $V(t)$ and $X(t)$, respectively, as
\begin{align}\label{def:v(t)-process}
V(t)= \langle V,\delta_t\rangle\quad \text{and}\quad X(t)= \langle V,1_{[0,t]}\rangle.
\end{align}
We now turn to the derivation of weak solutions for the GLE. By formally multiplying the GLE \eqref{eq:gle} by a test function $\f$ in $\Sc$ and integrating by parts, we arrive at the integral equation
\begin{displaymath}
\begin{aligned}
	- m\int_\rbb V(t)\varphi'(t)\d t &=  - \beta \int_\rbb V(t)\int_\rbb K^+(u)\varphi(t+u)\d u\d t + \sqrt{\beta} \int_\rbb F(t)\varphi(t)\d t,
\end{aligned}
\end{displaymath}
where
\begin{equation}\label{e:K+(t)}
K^+\!(t) := K(t) \, 1_{\{t\geq 0\}}.
\end{equation}
Then, for $F$ and $V$ as given by \eqref{e:F:S->L2} and \eqref{e:V:S'->L^2}, respectively, we obtain the weak form of~\eqref{eq:gle}, i.e.,
\begin{equation} \label{eq:gle-weak}
\langle V,-m\varphi'+\beta\widetilde{K^+*\widetilde{\varphi}}\rangle =\sqrt{\beta}\langle F,\varphi\rangle.
\end{equation}
%where $\widetilde{f}(x) := f(-x)$.
In this context, $F$ is understood as a stationary random distribution defined by means of the relation
\begin{displaymath}
\E{\langle F,\varphi_1\rangle\overline{\langle F,\varphi_2\rangle}} = \int_\rbb K(t)\left( \varphi_1*\widetilde{\varphi}_2 \right)(t)\d t.
\end{displaymath}
In view of Proposition~\ref{prop:temper-distribution:Fcos}, for the memory kernel $K$, we have
\begin{displaymath}
\E{\langle F,\varphi_1\rangle\overline{\langle F,\varphi_2\rangle}} = \int_\rbb K(t)\left( \varphi_1*\widetilde{\varphi}_2 \right)(t)\d t=\int_\rbb2\Kcos(\omega)\widehat{\f_1}(\omega)\overline{\widehat{\f_2}(\omega)}\d\omega.
\end{displaymath}
In particular, the spectral measure of $F$ is $2\Kcos(\omega)\d\omega$.

%Having introduced the needed terminology, we have the following definition of a stationary solution of~\eqref{eq:gle}, which is a shorter version of ~\cite[Definition 4.1]{mckinley2018anomalous}
We are now in a position to provide the definition of a stationary solution of~\eqref{eq:gle} (cf.\ ~\cite[Definition 4.1]{mckinley2018anomalous}).
\begin{definition}\cite{mckinley2018anomalous} \label{def:weak-solution}
Let $\nu$ be a nonnegative measure satisfying condition \eqref{ineq:spectral-measure} and let $V$ be the operator associated with $\nu$ defined in \eqref{eqn:V.op}. Also, consider $Dom(V)$ and $K^{+}(t)$ as defined by \eqref{e:Dom(V)} and \eqref{e:K+(t)}, respectively. Then, $V$ is a \textit{weak solution} of \eqref{eq:gle} if the following conditions are satisfied.
\begin{enumerate}[(a)]
\item For every $\varphi\in\Sc$, $K^+*\varphi$ belongs to $\text{Dom}(V)$;
\item for any $\varphi,\psi\in \Sc$,
\begin{align*}
\Enone\Big[\langle V,-m\varphi'+\beta\widetilde{K^+*\widetilde{\varphi}}\rangle\overline{\langle V,-m\psi'+\beta\widetilde{K^+*\widetilde{\psi}}\rangle}\Big]=\Enone\big[\langle\sqrt{\beta}F,\varphi\rangle \overline{\langle \sqrt{\beta}F,\psi\rangle}\big].\end{align*}
\end{enumerate}
\end{definition}

Bearing in mind the above definition of a weak solution, we can now state and establish the well-posedness of~\eqref{eq:gle}.
\begin{theorem} \label{thm:well-posed:critical} Suppose that $K(t)$ satisfies~\eqref{cond1} and~\eqref{cond:powerlaw:t}. Then, $V$ is a weak solution for \eqref{eq:gle} (see Definition~\ref{def:weak-solution}) if and only if the spectral measure $\nu$ satisfies $\nu(\d\omega)=\widehat{r}(\omega)\d\omega$, where $\widehat{r}$ is given by
\begin{equation} \label{eqn:spectral-density}
\widehat{r}(\omega) :=\frac{ \beta\widehat{K}(\omega)}{2\pi\lvert mi\omega+\beta\widehat{K^+}(\omega)\rvert^2}.
\end{equation}
\end{theorem}
\begin{remark} Formula~\eqref{eqn:spectral-density} is also the spectral density of the weak solutions in diffusive and subdiffusive regimes \cite{mckinley2018anomalous}.
\end{remark}
\begin{proof} First, we claim that $\rhat$ as given by~\eqref{eqn:spectral-density} is integrable. In fact, we can recast this expression as
\begin{equation} \label{eqn:spectral-density:Kcos-Ksin}
\widehat{r}(\omega) = \frac{1}{2\pi}\hspace{0.5mm}\frac{2\beta\Kcos(\omega)}
{\left[\beta\Kcos(\omega) \right]^2
+\left[m\omega- \beta\Ksin(\omega) \right]^2}.
\end{equation}
Note that $\rhat(\omega)$ is well-defined, since, by condition~\eqref{cond1c}, $\Kcos(\omega)$ is assumed to be strictly positive for every $\omega>0$. Moreover, it is symmetric around zero since the memory kernel $K$ is also so by condition~\eqref{cond1a}. By virtue of Lemma~\ref{lem:fcos-fsin:well-defined}, $\rhat(\omega)$ is continuous for $\omega\in(0,\infty)$. Therefore, we only need to check integrability at $\omega\to\infty$ and around the origin. On one hand, as $\omega\to\infty$, Lemma~\ref{lem:fcos-fsin:well-defined} implies that $\Kcos(\omega)$ and $\Ksin(\omega)$ converge to zero. It follows that $\rhat(\omega)$ is dominated by $\omega^{-2}$. On the other hand, %when $\omega$ is near the origin, we have
\begin{align*}
\rhat(\omega)\leq \frac{1}{\pi\beta\Kcos(\omega)}\to 0,\qquad\omega\to 0.
\end{align*}
By virtue of Proposition~\ref{prop:asymptotic:Fcos-sin:critical}, $\Kcos(\omega)\sim |\log(\omega)|$ as $\omega\to 0$. Therefore, $\rhat$ is integrable, as claimed.

In light of Proposition~\ref{prop:temper-distribution:Fcos}, the remaining claims can be established by a simple adaptation of the proof of \cite[Theorem 4.3]{mckinley2018anomalous}.
\end{proof}
%\begin{remark} The formula of the spectral density $\rhat$ implicitly requires that $\Kcos(\omega)\geq 0$
%\end{remark}

In the last result of this section, we characterize the sample path regularity of the velocity process $V(t)$. Its proof is analogous to that of \cite[Theorems 5.4 and 5.6]{mckinley2018anomalous}, and thus is omitted.

\begin{proposition} \label{prop:regularity}
Under the assumptions of Theorem \ref{thm:well-posed:critical}, let $V(t)$ be the process defined in ~\eqref{def:v(t)-process}.

\noindent(a) Then, there exists a modification $\widetilde{V}(t)$ of $V(t)$ such that $\widetilde{V}(t)$ is a.s.\ continuous.

\noindent(b) Assume, further, that $K$ is a positive definite function and that for some $b>3$
\begin{equation} \label{cond:differentiability}
|K(0)-K(t)|=O\big(|\log t |^{-b}\big),\quad\text{as}\quad t\to 0^+.
\end{equation}
Then, $\widetilde{V}(t)$ as in (a) is a.s.~continuously differentiable.
\end{proposition}

\begin{remark}\label{r:harmonizable}
Together, Theorem \ref{thm:well-posed:critical} and \cite[Theorem 4.3]{mckinley2018anomalous} establish the existence of a harmonizable representation
\begin{equation}\label{e:X(t)_harmonizable}
X(t) = \int_{\bbR} \frac{e^{it \omega}-1}{i \omega} \hspace{1mm}\widehat{r}^{1/2}(\omega) \widetilde{B}(d \omega), \quad t \geq 0,
\end{equation}
for the position particle associated with the GLE in all three regimes (critical, diffusive and subdiffusive). In \eqref{e:X(t)_harmonizable}, $\widetilde{B}(d \omega)$ is a $\bbC$-valued Gaussian random measure such that $\widetilde{B}(-d \omega)=\overline{\widetilde{B}(d \omega)}$ and $\bbE |\widetilde{B}(d \omega)|^2 = \theta \hspace{0.5mm}dx$ for some $\theta > 0$. Representations of the type \eqref{e:X(t)_harmonizable} have manifold uses in Probability theory (e.g., \cite{taqqu:samorodnitsky:1994,bierme:meerschaert:scheffler:2007}). In particular, a harmonizable representation of the form \eqref{e:X(t)_harmonizable} is the basis for the construction of the asymptotic distribution of the TAMSD for a broad class of anomalous diffusion models~\cite{didier2017asymptotic}. \end{remark}

\section{Asymptotics of the MSD in the critical regime} \label{sec:MSD:critical}

In this section, we establish the asymptotic behavior of the MSD when $K(t)\sim t^{-1}$ as $t\to\infty$. The approach is similar to that in Section 6 of \cite{mckinley2018anomalous}. For the reader's convenience, we summarize the method as follows.\vspace{0.3cm}

\noindent \textbf{step 1}: we use Proposition~\ref{prop:asymptotic:Fcos-sin:critical} to relate the large-time behavior of the memory $K$ to the near-zero behaviors of $\Kcos(\omega)$ and $\Ksin(\omega)$, i.e., as $\omega\to 0$;\vspace{0.1cm}

\noindent \textbf{step 2}: we obtain the near-zero behavior of the spectral density $\widehat{r}(\omega)$ as in~\eqref{eqn:spectral-density:Kcos-Ksin} through that of $\Kcos(\omega)$ and $\Ksin(\omega)$ as $\omega\to 0$;\vspace{0.1cm}

\noindent \textbf{step 3}: by the dominated convergence theorem and the near-zero behavior of $\rhat(\omega)$, we conclude that $\E{X(t)^2}\sim t/\log(t)$.

\begin{proof}[Proof of Theorem~\ref{thm:asymptotic:critical} ] Using the relation~\eqref{eqn:V.op} and~\eqref{def:v(t)-process}, we note that $\E{X^2(t)}$ can be written explicitly as
\begin{equation} \label{eqn:X(t)}
\begin{aligned}
\E{X(t)^2}=\int_\rbb \left|\widehat{1_{[0,t]}}(\omega)\right|^2\widehat{r}(\omega)\d\omega=4\int_0^\infty\!\! \frac{1-\cos(t\omega)}{\omega^2} \widehat{r}(\omega)\d\omega,
\end{aligned}
\end{equation}
since $\rhat$ is symmetric. It follows that
\begin{equation} \label{eqn:thm:asymptotic:critical:1}
\begin{aligned}
\frac{\log(t) \E{X(t)^2}}{t}&=\frac{4\log(t)}{t}\int_0^\infty\!\! \frac{1-\cos(t\omega)}{\omega^2} \widehat{r}(\omega)\d\omega\\
&=4\log(t)\int_0^\infty\!\! \frac{1-\cos(z)}{z^2}
\widehat{r}\left(\frac{z}{t}\right)\d z,
\end{aligned}
\end{equation}
where the second equality is a consequence of the change of variable $z:=t\omega$. Therefore, it suffices to show that the expression on the right-hand side of \eqref{eqn:thm:asymptotic:critical:1} has a finite and strictly positive limit as $t\to\infty$. In fact, we can split the integral on the right-hand side of \eqref{eqn:thm:asymptotic:critical:1} into two parts, i.e.,
\begin{align*}
\int^{\infty}_0\close \log(t)\frac{1-\cos(z)}{z^2}
\widehat{r}\left(\frac{z}{t}\right)\d z & = \Big\{\int_0^{\sqrt{t}}+\int_{\sqrt{t}}^\infty \Big\}\log(t) \frac{1-\cos(z)}{z^2}
\widehat{r}\left(\frac{z}{t}\right)\d z\\
&=I_1+I_2.
\end{align*}
Concerning $I_2$, recall from the proof of Theorem~\ref{thm:well-posed:critical} that $\rhat(\omega)$ is bounded for $\omega\in (0,\infty)$. Therefore, as $t\to\infty$,
\begin{align*}
0 \leq \hspace{1mm}I_2=\int_{\sqrt{t}}^\infty \log(t) \frac{1-\cos(z)}{z^2}
\widehat{r}\left(\frac{z}{t}\right)\d z&\leq c\log(t)\int_{\sqrt{t}}^\infty  \frac{1-\cos(z)}{z^2}\d z\\
&\leq c\frac{\log(t)}{\sqrt{t}}\to 0.
\end{align*}
With regards to $I_1$, by expression ~\eqref{eqn:spectral-density:Kcos-Ksin} for $\rhat$, we obtain
\begin{align*}
\log(t)\rhat\left(\frac{z}{t}\right)&=\frac{\log(t)}{2\pi}\hspace{1mm}\frac{2\beta\Kcos\left(\frac{z}{t}\right)}
{\left[\beta\Kcos\left(\frac{z}{t}\right) \right]^2
+\left[m\left(\frac{z}{t}\right)- \beta\Ksin\left(\frac{z}{t}\right) \right]^2}\\
&=\frac{\log(t)}{2\pi\log\left(\frac{t}{z}\right)}\hspace{1mm}\frac{2\beta\Kcos\left(\frac{z}{t}\right)/\log\left(\frac{t}{z}\right)}
{\left[\beta\Kcos\left(\frac{z}{t}\right)/\log\left(\frac{t}{z}\right) \right]^2
+\left[m\left(\frac{z}{t}\right)- \beta\Ksin\left(\frac{z}{t}\right) \right]^2/\log^2\left(\frac{t}{z}\right)}.
\end{align*}
Therefore, by Proposition~\ref{prop:asymptotic:Fcos-sin:critical},
\begin{align*}
\log(t)\rhat\left(\frac{z}{t}\right)\to \frac{1}{\pi\beta C_1}\in(0,\infty), \quad t\to\infty,
\end{align*}
where $C_1$ is given by \eqref{e:K:convergent-rate:critical}. Furthermore, assuming $t$ is sufficiently large, for every $z\in(0,\sqrt{t}]$, we have the uniform bound
\begin{align*}
\MoveEqLeft[7]
\frac{\log(t)}{\log\left(\frac{t}{z}\right)}\hspace{1mm}\frac{2\beta\Kcos\left(\frac{z}{t}\right)/\log\left(\frac{t}{z}\right)}
{\left[\beta\Kcos\left(\frac{z}{t}\right)/\log\left(\frac{t}{z}\right) \right]^2
+\left[m\left(\frac{z}{t}\right)- \beta\Ksin\left(\frac{z}{t}\right) \right]^2/\log^2\left(\frac{t}{z}\right)}&\\
&\leq \frac{\log(t)}{\log\left(\frac{t}{z}\right)}\hspace{1mm} \frac{2}
{\beta\Kcos\left(\frac{z}{t}\right)/\log\left(\frac{t}{z}\right)}\\
&\leq \frac{\log(t)}{\log(\sqrt{t})}\hspace{1mm} \frac{c}
{\Kcos\left(\frac{z}{t}\right)/\log\left(\frac{t}{z}\right)}\\
&\leq c<\infty,
\end{align*}
since $\Kcos(\omega)\sim|\log(\omega)|$ as $\omega\to 0$, by virtue of Proposition~\ref{prop:asymptotic:Fcos-sin:critical}. The dominated convergence theorem then implies that
\begin{align*}
I_1=\int_0^{\sqrt{t}}\close \log(t) \frac{1-\cos(z)}{z^2}
\widehat{r}\left(\frac{z}{t}\right)\d z \to \frac{1}{\pi\beta C_1}\int_0^\infty  \frac{1-\cos(z)}{z^2}
\d z\in(0,\infty), \,\text{as}\,\, t\to\infty.
\end{align*}
The result now follows from combining the asymptotics of $I_1$ and $I_2$. The proof is thus complete.

\end{proof}

\section{Robust bounds for the asymptotic behavior of the MSD} \label{sec:MSD:convergent-rate}

In this section, we construct robust bounds on the deviation of the MSD from its asymptotic trend in all three different regimes. By analogy to Section~\ref{sec:MSD:critical}, the general procedure is based on obtaining bounds for the convergence rate of the spectral density $\rhat(\omega)$ as $\omega\to 0$. We begin by stating and showing the following auxiliary result, which is used in the proof of the subsequent Theorem~\ref{thm:convergent-rate:diffusion+subdiffusion}. Note that expression~\eqref{eqn:spectral-density:Kcos-Ksin} for $\rhat(\omega)$ holds in the three regimes (cf.\ \cite[expression (65)]{mckinley2018anomalous}).
\begin{proposition} \label{prop:convergent-rate:diffusion+subdiffusion}
Suppose that $K(t)$ satisfies~\eqref{cond1}. Let $\rhat(\omega)$ be the spectral density function given by~\eqref{eqn:spectral-density:Kcos-Ksin}.

\noindent (a) If $K(t)$ satisfies~\eqref{cond:K:convergent-rate:diffusion}, then
\begin{align*}
\Big|\rhat(\omega)-\frac{1}{\pi\beta\Kcos(0)}\Big|=O(\omega^{\gamma_0}),\quad\text{as}\quad\omega\to 0,
\end{align*}
where $\gamma_0=\min\{\beta_0,2\}$ and $\beta_0$ is the constant from~\eqref{cond:K:convergent-rate:diffusion};

\noindent (b) if $K(t)$ satisfies~\eqref{cond:K:convergent-rate:subdiffusion}, then
\begin{align*}
\Big|\frac{\rhat(\omega)}{\omega^{1-\alpha}}- \frac{\int_0^\infty\frac{\cos(z)}{z^\alpha}\d z}{\pi\beta C_\alpha\big[ \big(\int_0^\infty\frac{\cos(z)}{z^\alpha}\d z\big)^2+\big(\int_0^\infty\frac{\sin(z)}{z^\alpha}\d z\big)^2\big]}\Big|=O(\omega^{\gamma_\alpha}),\quad\text{as}\quad\omega\to 0,
\end{align*}
where $C_\alpha=\lim_{t\to\infty}t^{\alpha}K(t)$ (see \eqref{e:K:convergent-rate:subdiffusion}), $\gamma_\alpha=\min\{1-\alpha,\alpha\beta_\alpha\}$ and $\alpha,\,\beta_\alpha$ are the constants from~\eqref{cond:K:convergent-rate:subdiffusion}.
\end{proposition}
\begin{proof} (a) Using formula~\eqref{eqn:spectral-density:Kcos-Ksin}, we see that
\begin{align*}
\pi\rhat(\omega) - \frac{1}{\beta \Kcos(0)}& = \frac{\beta\Kcos(\omega)}
{\left[\beta\Kcos(\omega) \right]^2
+\left[m\omega- \beta\Ksin(\omega) \right]^2}-\frac{1}{\beta \Kcos(0)}\\
&= \frac{\beta^2\Kcos(\omega)[\Kcos(0)-\Kcos(\omega)]+\left[m\omega-\beta\Ksin(\omega) \right]^2}
{\beta\Kcos(0)\left(\left[\beta\Kcos(\omega) \right]^2
+\left[m\omega-\beta\Ksin(\omega) \right]^2\right)}
\end{align*}
In view of Lemma~\ref{lem:fcos-fsin:well-defined}, as $\omega\to 0$, $\Kcos(\omega)$ converges to $\Kcos(0)=\int_0^\infty K(t)\d t>0$. It follows that for $\omega>0$ sufficiently small,
\begin{align*}
\Big|\pi\rhat(\omega) - \frac{1}{\beta \Kcos(0)}\Big|\leq c |\Kcos(0)-\Kcos(\omega)|+c\left[m\omega-\beta\Ksin(\omega) \right]^2.
\end{align*}
We now invoke Lemma~\ref{lem:Kcos:convergent-rate:diffusion} to obtain
\begin{align*}
\Big|\pi\rhat(\omega) - \frac{1}{\beta \Kcos(0)}\Big|\leq c(\omega^{\gamma_{0}}+\omega^2+\omega^{2\gamma_{0,1}})=O(\omega^{\gamma_0}),
\end{align*}
since $\gamma_0=\min\{\beta_0,2\}\leq 2\gamma_{0,1}=2\min\{\beta_0,1\}\leq 2$ as in Lemma~\ref{lem:Kcos:convergent-rate:diffusion}. This concludes the proof of part (a).

In regard to part (b), to simplify the notation we set
\begin{align} \label{eqn:C-alpha-cos-sin}
C_{\alpha,1} = C_\alpha\int_0^\infty\frac{\cos(z)}{z^\alpha}\d z,\quad\text{and}\quad C_{\alpha,2} = C_\alpha\int_0^\infty\frac{\sin(z)}{z^\alpha}\d z.
\end{align}
We note that since $z^{-\alpha}$ is concave up and decreasing on $(0,\infty)$, two integrals above are positive (see~\cite{tuck2006positivity}) and so are $C_{\alpha,1}$ and $C_{\alpha,2}$. Then, using formua~\eqref{eqn:spectral-density:Kcos-Ksin} again, we have
\begin{multline*}
\frac{\pi\rhat(\omega)}{\omega^{1-\alpha}}-\frac{\int_0^\infty\frac{\cos(z)}{z^\alpha}\d z}{\beta C_\alpha\big[ \big(\int_0^\infty\frac{\cos(z)}{z^\alpha}\d z\big)^2+\big(\int_0^\infty\frac{\sin(z)}{z^\alpha}\d z\big)^2\big]}
\\ = \frac{\beta\omega^{1-\alpha}\Kcos(\omega)}
{\left[\beta\omega^{1-\alpha}\Kcos(\omega) \right]^2
+\left[m\omega^{2-\alpha}- \beta\omega^{1-\alpha}\Ksin(\omega) \right]^2}-\frac{C_{\alpha,1}}{\beta (C_{\alpha,1}^2+C_{\alpha,2}^2)}.
\end{multline*}
After subtraction, the numerator of the right-hand side above is written as
\begin{multline*}
\beta^2\big[C_{\alpha,1}\omega^{1-\alpha}\Kcos(\omega) -C_{\alpha,2}^2\big]\big[C_{\alpha,1}-\omega^{1-\alpha}\Kcos(\omega)\big]\\-C_{\alpha,1}\omega^{2-\alpha}\big[m^2\omega^{2-\alpha}-2m\beta\omega^{1-\alpha}\Ksin(\omega)\big]
+\beta^2C_{\alpha,1}\big[C_{\alpha,2}^2-(\omega^{1-\alpha}\Ksin(\omega))^2\big].
\end{multline*}
In view of Lemma~\ref{lem:Kcos:convergent-rate:subdiffusion}, as $\omega\to 0$, $\omega^{1-\alpha}\Kcos(\omega)$ and $\omega^{1-\alpha}\Ksin(\omega)$ converge to $C_{\alpha,1}$ and $C_{\alpha,2}$, respectively. Similar to part (a), we arrive at the following estimate
\begin{multline*}
\frac{\pi\rhat(\omega)}{\omega^{1-\alpha}}-\frac{\int_0^\infty\cos(z)/z^\alpha\d z}{\beta C_\alpha\int_0^\infty(\cos(z)+\sin(z))/z^\alpha\d z}\\ \leq c|\omega^{1-\alpha}\Kcos(\omega)-C_{\alpha,1}|+c|\omega^{1-\alpha}\Ksin(\omega)-C_{\alpha,2}|+O(\omega^{2-\alpha}),
\end{multline*}
whence
\begin{align*}
\frac{\pi\rhat(\omega)}{\omega^{1-\alpha}}-\frac{\int_0^\infty\cos(z)/z^\alpha\d z}{\beta C_\alpha\int_0^\infty(\cos(z)+\sin(z))/z^\alpha\d z}=O(\omega^{\gamma_\alpha}+\omega^{2-\alpha})=O(\omega^{\gamma_\alpha}),
\end{align*}
where $0<\gamma_\alpha<2-\alpha$ is the constant from Lemma~\ref{lem:Kcos:convergent-rate:subdiffusion}. The proof is thus complete.
\end{proof}

We are now in a position to prove Theorem~\ref{thm:convergent-rate:diffusion+subdiffusion}.

\begin{proof}[Proof of Theorem~\ref{thm:convergent-rate:diffusion+subdiffusion}] We first show (a). By making the change of variable $z=t\omega$, recast expression~\eqref{eqn:X(t)} as
\begin{align} \label{eqn:thm:convergent-rate:diffusion+subdiffusion:1}
\frac{\E{X(t)^2}}{t} = 4\int_0^\infty \frac{1-\cos(z)}{z^2}\rhat\Big(\frac{z}{t}\Big)\d z.
\end{align}
Therefore, for sufficiently small $\epsilon>0$ and large enough $t$,
\begin{align} \label{eqn:thm:convergent-rate:diffusion+subdiffusion:1-three_integ}
\frac{\E{X(t)^2}}{t}-\frac{4}{\pi\beta\Kcos(0)}&\int_0^\infty\frac{1-\cos(z)}{z^2} = 4\int_0^\infty \frac{1-\cos(z)}{z^2}\Big[\rhat\Big(\frac{z}{t}\Big)-\frac{1}{\pi\beta\Kcos(0)}\Big]\d z \nonumber \\
&=4 \Big\{\int_0^1+\int_1^{\epsilon t}+\int_{\epsilon t}^\infty \Big\} \Big[\rhat\Big(\frac{z}{t}\Big)-\frac{1}{\pi\beta\Kcos(0)}\Big]\d z.
\end{align}
We now construct bounds for each integral term on the right-hand side of \eqref{eqn:thm:convergent-rate:diffusion+subdiffusion:1-three_integ}. In view of the proof of Theorem 6.1 \cite[p.\ 5149]{mckinley2018anomalous}, when $K$ is integrable, $\rhat(\omega)$ is bounded on $(0,\infty)$. It follows that
\begin{align}\label{e:int^infty-eps*t}
\Big| \int_{\epsilon t}^\infty \frac{1-\cos(z)}{z^2}\Big[\rhat\Big(\frac{z}{t}\Big)-\frac{1}{\pi\beta\Kcos(0)}\Big]\d z \Big|\leq c\int_{\epsilon t}^\infty\frac{1}{z^2}\d z=O(t^{-1}).
\end{align}
By Proposition~\ref{prop:convergent-rate:diffusion+subdiffusion}, (a), and the fact that $(1-\cos(z))/z^2$ is bounded on $\rbb$, we obtain
\begin{align}\label{e:int^1-0}
\Big|\int_0^1 \frac{1-\cos(z)}{z^2}\Big[\rhat\Big(\frac{z}{t}\Big)-\frac{1}{\pi\beta\Kcos(0)}\Big]\d z\Big|\leq \frac{c}{t^{\gamma_0}}\int_0^1 z^{\gamma_0}\d z=O(t^{-\gamma_0}).
\end{align}
Likewise,
\begin{align}\label{e:int^esp*t-1}
\Big|\int_1^{\epsilon t}\frac{1-\cos(z)}{z^2}\Big[\rhat\Big(\frac{z}{t}\Big)-\frac{1}{\pi\beta\Kcos(0)}\Big]\d z \Big|\leq \frac{c}{t^{\gamma_0}}\int_1^{\epsilon t}\frac{1}{z^{2-\gamma_0}}\d z  \leq \frac{c}{t^{\gamma_0/2}},
\end{align}
where the last implication holds for any $\gamma_0\in[0,2]$. Expressions \eqref{eqn:thm:convergent-rate:diffusion+subdiffusion:1-three_integ}--\eqref{e:int^esp*t-1} imply that
\begin{align*}
\Big|\frac{\E{X(t)^2}}{t}-\frac{4}{\pi\beta\Kcos(0)}\int_0^\infty\frac{1-\cos(z)}{z^2}\d z \Big| =O(t^{-1}+t^{-\gamma_0}+t^{-\gamma_0/2})=O(t^{-\gamma_0/2}).
\end{align*}
Since, in addition,
\begin{equation}\label{e:integ_(1-cos)/z2}
\int_0^\infty  \frac{1-\cos(z)}{z^2} dz = \frac{\pi}{2}
\end{equation}
\cite[p.\ 447, (3.782.2)]{gradshteyn:ryzhik:2007}, then \eqref{e:MSD_III} holds. %which completes the proof of part (a).

To show part (b), on the subdiffusive regime, we employ the same technique as the one used in part (a). In this situation, by analogy to~\eqref{eqn:thm:convergent-rate:diffusion+subdiffusion:1}, we see that
\begin{align*}
\frac{\E{X(t)^2}}{t^\alpha} = 4\int_0^\infty \frac{1-\cos(z)}{z^{1+\alpha}}\hspace{1mm}\frac{\rhat\left(\frac{z}{t}\right)}{\left(\frac{z}{t}\right)^{1-\alpha}}\d z.
\end{align*}
As in the proof of part (a), fix a small $\epsilon > 0$ and a large enough $t$. Thus,
\begin{align}\label{e:t^alpha_three_integ}
\MoveEqLeft[6]
\frac{\E{X(t)^2}}{t^\alpha} -\frac{4C_{\alpha,1}}{\pi\beta(C_{\alpha,1}^2+C_{\alpha,2}^2)}\int_0^\infty\frac{1-\cos(z)}{z^{1+\alpha}}\d z \nonumber\\
&= 4\int_0^\infty \frac{1-\cos(z)}{z^{1+\alpha}}\bigg[\frac{\rhat\left(\frac{z}{t}\right)}{\left(\frac{z}{t}\right)^{1-\alpha}}-\frac{C_{\alpha,1}}{\pi\beta(C_{\alpha,1}^2+C_{\alpha,2}^2)}\bigg]\d z \nonumber\\
&=4\Big\{\int_0^1\close+\int_1^{\epsilon t}\close+\int_{\epsilon t}^\infty \Big\} \frac{1-\cos(z)}{z^{1+\alpha}}\bigg[\frac{\rhat\left(\frac{z}{t}\right)}{\left(\frac{z}{t}\right)^{1-\alpha}}-\frac{C_{\alpha,1}}{\pi\beta(C_{\alpha,1}^2+C_{\alpha,2}^2)}\bigg]\d z\nonumber\\
&=: 4(I_0+I_1+I_2)
\end{align}
We now provide bounds on each term on the right-hand side of \eqref{e:t^alpha_three_integ}. First note that, by Proposition~\ref{prop:convergent-rate:diffusion+subdiffusion}, (b),
\begin{align}\label{e:t^alpha_three_integ_I0}
I_0 =O(t^{-\gamma_\alpha}).
\end{align}
However, Proposition~\ref{prop:convergent-rate:diffusion+subdiffusion}, (b), also implies that $\rhat(\omega)/\omega^{1-\alpha}$ is bounded on $(0,\infty)$. By a similar argument to the one used in part (a), we readily obtain
\begin{align}\label{e:t^alpha_three_integ_I2}
|I_2|\leq c\int_{\epsilon t}^\infty\close\frac{1}{z^{1+\alpha}}\d z = O(t^{-\alpha}).
\end{align}
In addition, for $\gamma_\alpha=\min\{1-\alpha,\alpha\beta_\alpha\}$,
\begin{align}\label{e:t^alpha_three_integ_I1}
|I_1|\leq \frac{c}{t^{\gamma_\alpha}}\int_1^{\epsilon t}\close \frac{1}{z^{1+\alpha-\gamma_\alpha}}\d z =O(t^{-\gamma_\alpha/2}+t^{-\alpha/2}),
\end{align}
where the equality holds for any $\alpha,\,\gamma_\alpha\in(0,1)$. Expressions \eqref{e:t^alpha_three_integ}--\eqref{e:t^alpha_three_integ_I1} imply that
\begin{equation}\label{e:t^alpha:conclude}
\begin{aligned}
\MoveEqLeft[2]\Big|\frac{\E{X(t)^2}}{t^\alpha} -\frac{4C_{\alpha,1}}{\pi\beta(C_{\alpha,1}^2+C_{\alpha,2}^2)}\int_0^\infty\!\frac{1-\cos(z)}{z^{1+\alpha}}\d z \Big|\\
& = O(t^{-\alpha/2}+t^{-\gamma_\alpha/2})=O(t^{-\eta/2}),
\end{aligned}
\end{equation}
where $\eta=\min\{\alpha,\gamma_\alpha\}$. To simplify the limiting constant in \eqref{e:t^alpha:conclude}, consider again $C_{\alpha,1}$ and $C_{\alpha,2}$ as in \eqref{eqn:C-alpha-cos-sin}. From \cite[p.\ 460, (3.823)]{gradshteyn:ryzhik:2007},
$$
\int_0^\infty\!\frac{1-\cos(z)}{z^{1+\alpha}}\d z = -\Gamma(-\alpha) \cos\Big( \frac{\alpha \pi}{2}\Big),
$$
and from \cite[p.\ 10, (1)]{batemantable} and \cite[p.\ 68, (1)]{batemantable}, respectively,
$$
\int_0^\infty \frac{\cos(z)}{z^\alpha}\d z=\frac{\pi}{2\Gamma(\alpha)\cos(\alpha\pi/2)},\quad \int_0^\infty \frac{\sin(z)}{z^\alpha}\d z=\frac{\pi}{2\Gamma(\alpha)\sin(\alpha\pi/2)}.
$$
Then, by Euler's reflection formula,
\begin{align*}
\MoveEqLeft[2]\frac{4C_{\alpha,1}}{\pi\beta(C_{\alpha,1}^2+C_{\alpha,2}^2)}\int_0^\infty\!\frac{1-\cos(z)}{z^{1+\alpha}}\d z\\
&= \frac{\frac{4\pi}{2\Gamma(\alpha)\cos(\alpha\pi/2) }   }{ \pi\beta C_{\alpha}\big[ \big(\frac{\pi}{2\Gamma(\alpha)\cos(\alpha\pi/2) }\big)^2+\big(\frac{\pi}{2\Gamma(\alpha)\sin(\alpha\pi/2) }\big)^2  \big]  }\big(-\Gamma(-\alpha)\cos(\alpha\pi/2) \big)\\
&=\frac{2\sin(\alpha\pi)}{\alpha\pi\beta C_\alpha
},
\end{align*}
which is the constant appearing in expression \eqref{e:case_IV_MSD/t^(alpha)-const}. This establishes (b).
\end{proof}

We finish this section by providing the proof of Theorem~\ref{thm:convergent-rate:critical} in the critical regime.
\begin{proof}[Proof of Theorem~\ref{thm:convergent-rate:critical} ] We recall from~\eqref{eqn:thm:asymptotic:critical:1} that
\begin{align*}
\frac{\log(t) \E{X(t)^2}}{t}=4\log(t)\int_0^\infty\! \frac{1-\cos(z)}{z^2}
\widehat{r}\left(\frac{z}{t}\right)\d z,
\end{align*}
whence
\begin{align}\label{e:EX2/t/logt-const}
\MoveEqLeft[7]\frac{\log(t)\E{X(t)^2}}{t}-\frac{4}{\pi\beta C_1}\int_0^\infty  \frac{1-\cos(z)}{z^2}
\d z \nonumber\\
&=4\int_0^\infty\Big[\log(t)\rhat\left(\frac{z}{t}\right)-\frac{1}{\pi\beta C_1}\Big]\frac{1-\cos(z)}{z^2}\d z \nonumber\\
&= 4 \Big\{\int_0^{\log(t)^{-2}}\close\close+\int_{\log(t)^{-2}}^{\log(t)^2}+\int_{\log(t)^2}^\infty \Big\} \Big[\log(t)\rhat\left(\frac{z}{t}\right)-\frac{1}{\pi\beta C_1}\Big]\frac{1-\cos(z)}{z^2}\d z \nonumber\\
&=4(I_0+I_1+I_2).
\end{align}
We now construct bounds on each term on the right-hand side of \eqref{e:EX2/t/logt-const}. We first consider $I_0$ and $I_2$, as they are easier to handle compared with $I_1$. To derive a bound on $I_0$, recall from the proof of Theorem~\ref{thm:well-posed:critical} that $\rhat(\omega)$ is uniformly bounded on $(0,\infty)$. Then,
\begin{align}\label{e:EX2/t/logt-const_I0}
|I_0|&= \int_0^{\log(t)^{-2}}\close\close \Big|\log(t)\rhat\left(\frac{z}{t}\right)-\frac{1}{\pi\beta C_1}\Big|\frac{1-\cos(z)}{z^2}\d z \nonumber\\
&\leq c(\log(t)+1)\int_0^{\log(t)^{-2}}\close\close1\,\d z =O(\log(t)^{-1}).
\end{align}
Likewise, in regard to $I_2$,
\begin{align}\label{e:EX2/t/logt-const_I2}
|I_2|& \leq \int_{\log(t)^{2}}^\infty \Big|\log(t)\rhat\left(\frac{z}{t}\right)-\frac{1}{\pi\beta C_1}\Big|\frac{1-\cos(z)}{z^2}\d z \nonumber\\
&\leq c(\log(t)+1)\int_{\log(t)^{2}}^\infty\frac{1}{z^2}\d z =O(\log(t)^{-1}).
\end{align}
Turning to $I_1$, expression~\eqref{eqn:spectral-density:Kcos-Ksin} for $\widehat{r}(\omega)$ implies that

\begin{align}\label{e:log(t)r-hat(z/t)-const}
\log(t)\rhat\left(\frac{z}{t}\right)-\frac{1}{\pi\beta C_1}&=\frac{1}{\pi}\Big[\frac{\beta\log(t)\Kcos\left(\frac{z}{t}\right)}
{\left[\beta\Kcos\left(\frac{z}{t}\right) \right]^2
+\left[m\left(\frac{z}{t}\right)- \beta\Ksin\left(\frac{z}{t}\right) \right]^2}-\frac{1}{\beta C_1}\Big] \nonumber \\
&=\frac{1}{\pi\beta C_1}\Big[\frac{\beta^2 C_1\log(t)\Kcos\left(\frac{z}{t}\right)-\left[\beta\Kcos\left(\frac{z}{t}\right) \right]^2}
{\left[\beta\Kcos\left(\frac{z}{t}\right) \right]^2
+\left[m\left(\frac{z}{t}\right)- \beta\Ksin\left(\frac{z}{t}\right) \right]^2} \nonumber \\
&\qquad\qquad \qquad\qquad-\frac{\left[m\left(\frac{z}{t}\right)- \beta\Ksin\left(\frac{z}{t}\right) \right]^2}
{\left[\beta\Kcos\left(\frac{z}{t}\right) \right]^2
+\left[m\left(\frac{z}{t}\right)- \beta\Ksin\left(\frac{z}{t}\right) \right]^2}\Big].
\end{align}
However, Proposition~\ref{prop:asymptotic:Fcos-sin:critical} implies that $\limsup_{\omega\to 0}\Ksin(\omega)^2<\infty$ and $\Kcos(\omega)\sim|\log(\omega)|$ as $\omega\to 0$. Therefore, for every $z\in[\log(t)^{-2},\log(t)^2]$ and large enough $t$, the second term on the right-hand side of \eqref{e:log(t)r-hat(z/t)-const} is bounded in absolute value by
%\begin{multline*}
%\frac{\left[m\left(\frac{z}{t}\right)- \beta\Ksin\left(\frac{z}{t}\right) \right]^2}
%{\left[\beta\Kcos\left(\frac{z}{t}\right) \right]^2
%+\left[m\left(\frac{z}{t}\right)- \beta\Ksin\left(\frac{z}{t}\right) \right]^2}\\ \leq
$$
\frac{c} {\left|\Kcos\left(\frac{z}{t}\right)\right|^2}\leq \frac{c}{\left|\log\left(\frac{z}{t}\right)\right|^2}\leq \frac{c}{|\log(t)-2\log(\log(t))|^2}=O(|\log(t)|^{-2}).
$$
%\end{multline*}
To obtain a similar bound for the first term on the right-hand side of \eqref{e:log(t)r-hat(z/t)-const}, note that
\begin{align*}
\frac{\beta^2 C_1\log(t)\Kcos\left(\frac{z}{t}\right)-\left[\beta\Kcos\left(\frac{z}{t}\right) \right]^2}
{\left[\beta\Kcos\left(\frac{z}{t}\right) \right]^2
+\left[m\left(\frac{z}{t}\right)- \beta\Ksin\left(\frac{z}{t}\right) \right]^2}&\leq  \frac{\beta^2 C_1\log(t)\Kcos\left(\frac{z}{t}\right)-\left[\beta\Kcos\left(\frac{z}{t}\right) \right]^2}
{\left[\beta\Kcos\left(\frac{z}{t}\right) \right]^2}\\
&\leq \frac{\left|\Kcos\left(\frac{z}{t}\right) - C_1\log(t)\right|}
{\Kcos\left(\frac{z}{t}\right) }\\
&\leq \frac{\left|\Kcos\left(\frac{z}{t}\right) - C_1\log\left(\frac{t}{z}\right)\right|}{\Kcos\left(\frac{z}{t}\right) }+\frac{C_1|\log(z)|}{\Kcos\left(\frac{z}{t}\right)}.
\end{align*}
Again for $z\in[\log(t)^{-2},\log(t)^2]$ and large enough $t$, Proposition~\ref{prop:asymptotic:Fcos-sin:critical} implies that
\begin{align*}
\frac{C_1|\log(z)|}{\Kcos\left(\frac{z}{t}\right)}\leq \frac{c\log(z)}{\log\left(\frac{t}{z}\right)}\leq \frac{c\log(\log(t))}{\log(t)-2\log(\log(t))}=O(\log(t)^{-1}),
\end{align*}
%where in the first implication, we have employed the fact that $\Ksin(\omega)\sim|\log(\omega)|$ as $\omega\to 0$ from Proposition~\ref{prop:asymptotic:Fcos-sin:critical}.
Also, by Lemma~\ref{lem:Kcos:convergent-rate:critical},
\begin{align*}
\frac{\left|\Kcos\left(\frac{z}{t}\right) - C_1\log\left(\frac{t}{z}\right)\right|}{\Kcos\left(\frac{z}{t}\right) }\leq \frac{c}{\Kcos\left(\frac{z}{t}\right) }=\frac{c}{\log\left(\frac{t}{z}\right)} \leq \frac{c}{\log(t)-2\log(\log(t))}=O(\log(t)^{-1}).
\end{align*}
Therefore,
\begin{align}\label{e:EX2/t/logt-const_I1}
|I_1 |&= \Big|\int_{\log(t)^{-2}}^{\log(t)^2}\Big[\log(t)\rhat\left(\frac{z}{t}\right)-\frac{1}{\pi\beta C_1}\Big]\frac{1-\cos(z)}{z^2}\d z \Big| \nonumber\\
&\leq \frac{c}{\log(t)}\int_{\log(t)^{-2}}^{\log(t)^2}\frac{1-\cos(z)}{z^2}\d z=O(\log(t)^{-1}).
\end{align}
Expressions \eqref{e:EX2/t/logt-const}--\eqref{e:EX2/t/logt-const_I1} imply that
\begin{align}\label{e:log(t)EX^2/t-const}
\Big|\frac{\log(t)\E{X(t)^2}}{t}-\frac{4}{\pi\beta C_1}\int_0^\infty  \frac{1-\cos(z)}{z^2}
\d z \Big|=4|I_0+I_1+I_2|=O(\log(t)^{-1}),
\end{align}
as $t\to\infty$. Relations \eqref{e:log(t)EX^2/t-const} and \eqref{e:integ_(1-cos)/z2} establish \eqref{e:convergent-rate:critical}.
\end{proof}

\section{Conclusion}\label{sec:discuss}

The GLE is a universal model for particle velocity in a viscoelastic medium. In this paper, we consider the GLE with power law decay memory kernel. We show that, in the critical regime where the memory kernel decays like $1/t$ as $t \rightarrow \infty$, the MSD of particle motion grows linearly in time up to a slowly varying (logarithm) term. Moreover, we use the theory of stationary random distributions to establish the well-posedness of the GLE in this regime. This solves an open problem from~\cite{mckinley2018anomalous} and completes the answer to the conjecture put forward in~\cite{morgado2002relation} on the relationship between memory kernel decay and anomalously diffusive behavior. Under slightly stronger assumptions on the memory kernel, we construct an Abelian-Tauberian framework to provide robust bounds on the deviation of the MSD around its asymptotic trend. This bridges the gap between the GLE memory kernel and the spectral density of anomalously diffusive particle motion characterized in \cite{didier2017asymptotic}.

The work in this paper leads to a number of future research directions. As mentioned in~\cite{mckinley2018anomalous}, it is an open question whether conditions such as \eqref{cond1} and~\eqref{cond:powerlaw:t} are not only sufficient, but also necessary for characterizing the growth rate of the MSD. Although sufficient and necessary conditions on the relationship between the memory kernel $K$ and its Fourier transforms $\Kcos$ and $\Ksin$ are fully provided in Propositions~\ref{prop:asymptotic:Fcos-sin:critical} and~\ref{prop:asymptotic:Fcos-sin:critical:Tauberian}, it remains an open problem to construct analogous necessary conditions for $\Kcos$, $\Ksin$ vis-\`{a}-vis the spectral density $\rhat$ in~\eqref{eqn:spectral-density:Kcos-Ksin}, or for $\rhat$ vis-\`{a}-vis the MSD $\E{X(t)^2}$ in~\eqref{eqn:X(t)}.

A related research topic that is of direct interest for experimental data analysis is to establish the asymptotic distribution of the time-averaged mean squared displacement statistic under the three GLE regimes by drawing upon the analytical framework developed in this paper. This would clarify or extend the connection between the GLE and the results in~\cite{didier2017asymptotic}, and is the topic of future work.

%We do not have results on the convergence of the process $X(t)$ in any other means, namely weak convergence as studied in~\cite{didier2017asymptotic} for the processes therein. It therefore remains an open question whether $X(t)$ weakly converges to any limiting under appropriate scaling.

\section*{Acknowledgements}  We would like to thank Scott A.\ McKinley for suggesting the problems. H.N.\ gratefully acknowledges support through the NSF DMS-1612898 grant. G.D.\ was partially supported by the prime award no.\ W911NF--14--1--0475 from the Biomathematics subdivision of the Army Research Office, USA. We are also grateful to two anonymous reviewers for their comments and suggestions.

\bibliography{power-law-bib}
\bibliographystyle{plain}

\end{document}